\documentclass[11pt]{article}

\usepackage{fullpage}

\usepackage{amsmath,amssymb,amsthm}
\usepackage{graphicx}
\usepackage[mathscr]{eucal}
\usepackage{pictexwd,dcpic}
\usepackage{stmaryrd}
\usepackage{enumerate}

\usepackage[pagewise]{lineno}
\usepackage{xcolor}

\usepackage{hyperref} 
\hypersetup{
    colorlinks,
    citecolor=red,
    filecolor=red,
    linkcolor=blue,
    urlcolor=red
}

\input xy
\xyoption{all}
\xyoption{2cell}
\UseAllTwocells
\input{diagxy}

\newtheorem{thm}{Theorem}[section]
\newtheorem{lem}[thm]{Lemma}
\newtheorem{prop}[thm]{Proposition}

\newtheorem{defn}[thm]{Definition}
\newtheorem{ex}[thm]{Example}
\newtheorem{rem}[thm]{Remark}

\def\tov{\nrightarrow}

\def\ot{\otimes}
\def\op{\oplus}
\def\rarr{\rightarrow}
\def\lrarr{\longrightarrow}
\def\rTo{\mbox{$ -\!\!\!\!\!-\!\!\circ$}}
\def\lTo{\mbox{$\circ\!\!-\!\!\!\!\!-$}}
\def\qzero{\mbox{\bf{0}}}
\def\qone{\mbox{\bf{1}}}
\renewcommand{\max}{\mbox{\textsf{max}}}
\renewcommand{\min}{\mbox{\textsf{min}}}
\newcommand{\ttop}{\mbox{$\top\!\!\!\top$}}
\newcommand{\bott}{\mbox{$\bot\!\!\!\bot$}}

\newcommand{\ZZP}{\mbox{$\mathbb{Z}_\infty$}}
\newcommand{\sREL}{\mbox{$\sf Rel$}}
\newcommand{\sORD}{\mbox{$\sf Ord$}}

\def\Mon{\mbox{${\sf Mon}$}}
\def\Matr{\mbox{${\sf Matr}$}}
\def\Mod{\mbox{${\sf Mod}$}}

\newcommand{\QRel}{\mbox{$Q$-${\sf Rel}$}}
\newcommand{\QMod}{\mbox{$\cQ$-${\sf Mod}$}}
\newcommand{\MonQ}{\mbox{${\sf Mon}\cQ$}}
\newcommand{\MatrQ}{\mbox{${\sf Matr}\cQ$}}

\newcommand{\BMod}{\mbox{$\cB$-${\sf Mod}$}}
\newcommand{\MonB}{\mbox{${\sf Mon}\cB$}}
\newcommand{\MatrB}{\mbox{${\sf Matr}\cB$}}

\newcommand{\LD}{{\sf LD}}

\def\cB{\mbox{$\mathcal{B}$}}
\def\cC{\mbox{$\mathcal{C}$}}
\def\cD{\mbox{$\mathcal{D}$}}
\def\cL{\mbox{$\mathcal{L}$}}
\def\cQ{\mbox{$\mathcal{Q}$}}
\def\cM{\mbox{$\mathcal{M}$}}
\def\cV{\mbox{$\mathcal{V}$}}

\def\ZZ{\mathbb{Z}}
\def\biquant{\mathcal{Q}\mathrm{uant}}
\def\biloc{\mathcal{L}\mathrm{oc}}
\def\qtld{\mathcal{Q}\mathrm{tld}}
\def\vmat{\mathcal{V}$-$\mathcal{M}\mathrm{at}}

\def\prof{\mbox{$\mathcal{P}\mathrm{rof}$}}

\newcommand{\Ladj}{\mbox{\small $\mathrel{\relbar\joinrel\parallel}$}}
\newcommand{\CLadj}{\mbox{\small $\mathrel{\parallel\joinrel\shortrightarrow\mkern-5mu\joinrel\relbar\joinrel\parallel}$}}
\newcommand{\coproduct}{\rotatebox[origin=c]{180}{$\prod$}}

\title{Constructing linear bicategories}
\author{\begin{tabular}[t]{c}
        Richard Blute\thanks{Research supported in part by
NSERC.} \,\,\,\,\,\,\,\,\,\,\,\,\,\,\,\,\,\,\,\,\,\,\,\,\,\,\,\,\,   Rose Kudzman-Blais*\\
       {\small Department of Mathematics and Statistics}\\ [-4pt]
        {\small University of Ottawa}\\  [-4pt]
        {\small Ottawa, Ontario, Canada}\\ [-4pt]
        \\
Susan Niefield\\
        {\small Department of Mathematics}\\ [-4pt]
        {\small Union College}\\  [-4pt]
        {\small Schenectady, NY, USA}\\ [-4pt]
        \end{tabular}}

\begin{document}

\maketitle

\begin{center}{\em Dedicated to the memory of our friend Phil Scott}\end{center}

\bigskip

\begin{abstract} {\it Linearly distributive categories} were introduced to model the tensor/par fragment of linear logic, without resorting to the use of negation. {\it Linear bicategories} are the bicategorical version of linearly distributive categories. Essentially, a linear bicategory has two forms of composition, each determining the structure of a bicategory, and the two compositions are related by a linear distribution. After the initial paper on the subject, there was little further work as there seemed to be a lack of examples. The main goal of this paper is to demonstrate that there are in fact a great many examples, which are obtained by considering quantales and quantaloids, and by extending familiar constructions from the (ordinary) bicategorical setting. It is standard in the field of {\it monoidal topology} that the category of quantale-valued relations is a bicategory. Here we begin by showing that a quantale is {\it Girard} if and only if the corresponding bicategory is a Girard quantaloid, which is an example of linear bicategory. The {\it tropical} and {\it arctic semiring} structures fit together into a Girard quantale, so this construction is likely to have multiple applications. More generally, we define \LD-quantales, which are sup-lattices with two quantale structures related by a linear distribution, and their bicategorical analogue, linear quantaloids. We show that \QRel\ is a linear quantaloid if and only if $Q$ is an \LD-quantale. We then consider several standard constructions from enriched bicategory theory, and show that these lift to the linear quantaloid setting and produce new examples of linear bicategories. In particular, we consider linear \cQ-categories, matrices in \cQ\ and linear monads in \cQ, where \cQ\ is a linear quantaloid. We develop non-locally posetal examples as well, $\biquant$, the bicategory of quantales, modules and module homomorphisms, and $\qtld$, the bicategory of quantaloids, modules and module homomorphisms. These turn out to be cyclic $*$-autonomous bicategories, which are in essence a closed version of linear bicategories. 
\end{abstract}

\newpage

\section{Introduction}

The idea behind the theory of linearly distributive categories (LDC) as introduced by \cite{Cockett_Seely_1997} is that (the multiplicative fragment of) linear logic is best modeled by taking the two multiplicative connectives, $\ot$ (tensor) and $\op$ (par), as primitive. One then obtains a category with two monoidal structures, related by {\it linear distributions}. A linear distribution is a natural transformation, not an isomorphism typically, of the following form or a symmetric equivalent:
\[A\ot (B\op C)\rarr (A\ot B)\op C\]
This approach to the model theory of linear logic differs from the original approach, using the $*$-autonomous categories of \cite{Barr_1979}, of taking tensor and negation as primitive and then defining the par by de Morgan duality.\\

Just as monoidal categories can be viewed as one-object bicategories, one can ask for the bicategorical version of linearly distributive categories. These are the {\it linear bicategories} of \cite{Cockett_Koslowski_Seely_2000}, which provide a natural semantics for non-commutative linear logic. The primary goal of this paper is to give new classes of linear bicategories arising from several different sources. \\ 

Before diving into the formal definitions, we describe the example that led to the more general discussion below. The structures we consider are two ordered semiring structures on the extended integers $\ZZP=\ZZ\cup\{+\infty,-\infty\}$, as examined by \cite{Golan_2003}. (One could just as well consider the extended reals.) This set in fact has two semiring structures and these are typically called the {\it tropical} and {\it arctic} semirings. They are of great use in the theory of synchronization as considered in \cite{Baccelli_Cohen_Olsder_Quadrat_1992}. See \cite{Droste_Kuich_2009, Droste_Kuich_Vogler_2009} for how extensively these structures arise. \\

In both, multiplication is given by the usual addition of integers. But we must be careful in defining $\infty+-\infty$.  In the first structure, we define $-\infty+_1\infty=-\infty=\infty+_1-\infty$. The addition for this structure is given by \max. This gives \ZZP\ the structure of an ordered semiring with \ZZP\ equipped with its usual order.\\

For the second structure, we again have that the multiplication is given by addition. But analogously we now must have $-\infty+_2\infty=\infty=\infty+_2-\infty$. The addition for this structure is given by \min. This gives \ZZP\ the structure of an ordered semiring with \ZZP\ equipped with the opposite of its usual order. \\

If $X$ and $Y$ are sets, we will define a \ZZP-{\it relation }from $X$ to $Y$ to be a function $R\colon X\times Y\rarr\ZZP$. As in the 
category of relations, we consider this as a morphism $R\colon X\tov Y$. Then the two semiring constructions above allow us to define two distinct relational compositions. \\

Explicitly, given $X\stackrel{A}{\tov}Y \stackrel{B}{\tov}Z$, we define
\[A\ot B(x,z)=\bigvee_{y\in Y}(A(x,y)+_1B(y,z)) \mbox{\,\,\,\,\,\,\,\,\,\,\,\,\,\,\,\,\,\,\,\,}A\op B(x,z)=\bigwedge_{y\in 
Y}(A(x,y)+_2B(y,z))\]

We are of course using the fact that both $\ZZP$ and $\mathbb{Z}_\infty^{op}$ are not just semirings with the above structures, but are in
fact {\it quantales}. We are thus following the program of {\it monoidal topology} as described in \cite{Hofmann_Seal_Tholen_2014}.  \\

The current project began with the observation that the two compositions described above are related by a {\it linear distribution} and in fact determine a locally posetal linear bicategory. On the other hand, we have the observation that \ZZP\ with the above operations is a Girard quantale, as investigated by \cite{Yetter_1990, Rosenthal_1990} and the two structures are related by the Girard duality. \\

We first introduce the quantale analogue of linearly distributive categories, and the quantaloidal analogue of linear bicategories, in Section \ref{LDquantales_linearquantaloids}, which we call {\em \LD-quantales} and {\em linear quantaloids} respectively. These definitions will underlie most of our main results. All Girard quantales are \LD-quantales and all Girard quantaloids are linear quantaloids. \\

Our first result in Section \ref{QRel} is that the category \QRel\ is a Girard quantaloid, defined by \cite{Rosenthal_1992}, and therefore a locally posetal linear bicategory, if and only if $Q$ is a Girard quantale. This extends naturally to the case where $Q$ is an arbitrary \LD-quantale. We then provide concrete examples of \QRel\ as a linear quantaloid. \\

In Section \ref{Enriched_linearquantaloid}, following the general theory of enriching in a bicategory and the work of \cite{Rosenthal_1992}, we introduce the quantaloid \QMod\ whose 0-cells are \cQ-categories, 1-cells are \cQ-modules and 2-cells are point-wise inequalities, where \cQ\ is a Girard quantaloid. \QMod\ is then itself a Girard quantaloid. This leads to considering enrichment in a linear quantaloid \cQ\ and the introduction of the bicategory \QMod\ of linear \cQ-categories and linear \cQ-modules. It is shown to be a linear quantaloid if and only if \cQ\ is itself linear. This result proceeds by first proving the corresponding theorems for linear monads in \cQ\ and matrices in \cQ. Given these new constructions, we provide more examples of locally posetal linear bicategories, using the linear quantaloids presented in the previous section. \\

We finally develop non-locally posetal examples as well. This is done in Section \ref{non-posetal}. We begin by considering $\biloc$, the bicategory whose objects are locales, 1-cells are bimodules and two-cells are bimodule homomorphisms, which we use to illustrate a more general notion. This turns out to be what \cite{Cockett_Koslowski_Seely_2000} refer to as cyclic $*$-autonomous bicategories, which are linear bicategories.  We show that a number of classic examples of bicategories fit into this framework. In particular, the bicategories of quantales and of quantaloids (with their respective modules) are linear bicategories. 
 
\begin{rem} 
There are unfortunate notational conflicts between linear logic notation and the usual notation for ordered structures, as well as within the linear logic community.  We now give our choice for notation for the remainder of the paper, chosen to be in line with the notation of \cite{Cockett_Seely_1997} and \cite{Cockett_Koslowski_Seely_2000}. 

\begin{itemize}
    \item For partially ordered sets, we will denote the top element by $\qone$ and the bottom element by $\qzero$, if they exist.
    \item We will denote quantales by $Q$ and quantaloids by \cQ. 
    \item Composition of arrows will be written with diagrammatic ordering. 
    \item In a monoidal category, we will use $\ot$ to denote the tensor product and $\top$ to denote the unit, including in the case of quantales, as opposed to $\&$ and $1$ used in \cite{Rosenthal_1990}. Moreover, we will use $\ot$ to denote composition in a bicategory and $\top_{X}$ or $\ttop_X$ to denote the identity 1-cell on $X$, in particular for quantaloids, as opposed to $\circ$ and $i_X$ used in \cite{Rosenthal_1996}.
    \item In a $*$-autonomous or linearly distributive category, there is a second monoidal structure, which we will denote by $\op$, rather than the $\bindnasrepma$ of \cite{Girard_1987}. This includes in the case of Girard quantales, as opposed to $\sqcup$ and $d$ in \cite{Rosenthal_1990} and \LD-quantales. The unit of this second monoidal product will be denoted by $\perp$. $\op$ will also denote the second composition in the context of linear bicategories, in particular in Girard or linear quantaloids, and $\bot_{X}$ or $\bott_{X}$ will be the identity 1-cell on $X$. Note that $\bot$, $\bot_X$ and $\bott_X$ will also be used to denoted cyclic dualizing elements and families in Girard quantales and quantaloids. 
\end{itemize}

\end{rem}

\section{Preliminaries}

\subsection{Quantales and quantaloids}

See \cite{Niefield_Rosenthal_1988,Rosenthal_1990,Rosenthal_1996} for a detailed discussion about quantales and quantaloids. 

\begin{defn} 
\begin{itemize}

\item A {\em quantale} (coined by \cite{Mulvey_1986}) is a partially ordered set $Q$ with all suprema and an associative multiplication $\ot\colon Q\times Q\rarr Q$ such that for all subsets $P\subseteq Q$ and all elements $a\in Q$, we have 
\[\big(\bigvee P\big)\ot a=\bigvee_{p\in P}p\ot a\mbox{\,\,\,\,\,\,\,\,\, and \,\,\,\,\,\,\,\,\,\,}a\ot\big(\bigvee
P\big)=\bigvee_{p\in P}a\ot p\]
Note that $Q$ necessarily satisfies $a\ot{\qzero=\qzero=\qzero}\ot a$.

\item Since the operations $(-)\ot a$ and $a\ot(-)$ preserve all sups, they have right adjoints for all $a\in Q$, known as the left and right {\em residuations}. We denote them by $(-)\lTo a$ and $a\rTo(-)$ respectively and they are defined by:
\[ c\,\lTo a=\bigvee \{ b\in Q \,\vert\, b\ot a\leq c\} \quad\quad{\rm and}\quad\quad a\rTo\,c=\bigvee \{ b\in Q \,\vert\, a\ot b\leq c\} \]

\item An element $\top\in Q$ is a {\em unit} if for all $a\in Q, \top\ot a = a\ot\top$, in which case $Q$ is called {\em unital}. \\

\noindent The following definitions and result are due to \cite{Yetter_1990}:

\item An element $\bot\in Q$ is a {\em cyclic dualizing element} if for all $a\in Q$, we have 
\[\bot\,\lTo a=a\rTo\,\bot \quad\quad{\rm and}\quad\quad (a\rTo\,\bot)\rTo\,\bot=a\]
A {\em Girard quantale} is a pair $(Q, \bot)$ where $Q$ is a quantale and $\bot$ is a chosen 
cyclic dualizing element. We denote $a\rTo\,\bot$ as $a^\perp$. 

\item If $Q$ is a Girard quantale, it has a second multiplication defined by the linear logic version of de Morgan duality:

\begin{lem}\label{girardquantalelem} 
Let $(Q,\bot)$ be a Girard quantale, then it is unital with $\top = \bot^\perp$ and the operation $(-)^\perp$ is a contravariant isomorphism. $Q^{op}$ is thus a unital quantale with multiplication $$a\op b=(b^\perp\ot a^\perp)^\perp$$ and unit $\bot$. Evidently this operation satisfies:
\[(\bigwedge P)\op a=\bigwedge_{p\in P} p\op a \mbox{\,\,\,\,\,\,\,\,\, and \,\,\,\,\,\,\,\,\,\,}a\op(\bigwedge 
P)=\bigwedge_{p\in P} a\op p \] 
\end{lem}

\end{itemize}
\end{defn}

\begin{ex}
\begin{enumerate}

\item Every locale (or frame) $L$, i.e., complete lattice satisfying the infinite distributive law, $a\wedge (\bigvee b_i) = \bigvee (a\wedge b_i)~ \forall a, b_i\in L$, is a unital quantale with $\ot = \wedge$ and $\top=1$. As discussed by \cite{Niefield_Rosenthal_1988}, a quantale is a locale if and only if it is commutative, right-sided and idempotent. Following is a list of important locales we will consider in this paper:
    \begin{enumerate}
    \item Truth value two-chain $\Omega=\{\qzero, \qone\}$
    \item Totally ordered 3-chain $3=\{\qzero,1/2,\qone\}$, with residuations given by \[ c\,\lTo a= \begin{cases} \qone& \mbox{if } a\leq c \\ c &\mbox{if } a> c\end{cases} \]
    \item Extended real half-line with opposite ordering ${\sf P_{max}} = ([0,\infty]^{op}, {\rm max}, 0)$, with residuations defined by \cite{Lawvere_1973} to be \[c\,\lTo a =  \begin{cases} c& \mbox{if } a < c  \\ 0 &\mbox{if } a\geq c \end{cases} \]
    \item Lattice of open sets $\mathcal{O}(X)$ in a topological space $X$
    \end{enumerate}	
	
\item The set of relations $\sREL(X)$ on a set $X$ is a unital quantale with standard relational composition as its operation and with the diagonal relation $\Delta_X$ as its unit, i.e., given relations $R, S\colon X\tov X$, 
\[(x,x'')\in R\ot S \quad\mbox{if and only if}\quad \exists x' \mbox{\,\,\,\,\,}(x,x')\in R\,\, \mbox{\,and\,} \,\,(x',x'')\in S\]
	
	
\item The extended real half-line with opposite ordering can be equipped with other quantale structures. In particular, consider \mbox{${\sf P_{+}} = ([0,\infty]^{op}, +, 0)$}, with its operation standard addition extended by $a+\infty = \infty+a=\infty$. It is often called Lawvere's quantale of positive real numbers, as it was first introduced by \cite{Lawvere_1973}.  Residuation is given by ``truncated subtraction'': \[ c\,\lTo a = \begin{cases} c-a& \mbox{if } a\leq c < \infty  \\ 0 &\mbox{if } a\geq c \\ \infty &\mbox{if } a<c=\infty \end{cases} \] 
	
\item The unit interval with multiplication $([0,1],\cdot,1)$ is isomorphic to \mbox{${\sf P_{+}} = ([0,\infty]^{op}, +, 0)$} under the map $x\mapsto -{\sf ln}(x)$, a unital homomorphism of quantales (function preserving arbitrary sups, quantale operation $\ot$ and the unit $\top$) with inverse $y\mapsto {\sf exp}(-y)$. Residuations are given ``truncated division'', as defined by \cite{Hofmann_Reis_2013}:
\[ c\,\lTo a = \begin{cases} c/a& \mbox{if } 0\neq a> c \\ 1 &\mbox{otherwise } \end{cases} \] 
	
\end{enumerate}
\end{ex}

\begin{defn} 
\begin{itemize}
\item  A {\em quantaloid}, coined by \cite{Abramsky_Vickers_1993}, is a category \cQ\ enriched over the category of complete 
lattices and sup-preserving maps.

\item As in the case of quantales, for all arrows $f\colon a\rarr b\in \cQ$, the functors \mbox{$(-)\ot f\colon\cQ(a',a)\rarr\cQ(a',b)$} and $f\ot(-)\colon\cQ(b,b')\rarr\cQ(a,b')$ preserve all sups and thus have right adjoints, also known as residuations, denoted by $(-)\lTo f\colon\cQ(a',b)\rarr\cQ(a',a)$ and $f\rTo(-)\colon\cQ(a,b')\rarr\cQ(b,b')$ respectively. \\

\noindent The following definitions and result are due to \cite{Rosenthal_1992}: 

\item A family of 1-cells $\cD=\{\bot_a\colon a\rarr a\mid a\in\cQ\}$ is a {\it cyclic family} if $f\rTo~\bot_a= \bot_b~\lTo f$, for all $f\colon a\rarr b$, and let $f^\perp$ denote their common value. Then \cD is called a {\it cyclic dualizing family} if $f^{\perp\perp}=f$, for all $f$.  

\noindent A {\it Girard quantaloid} is a quantaloid $\cQ$ together with a {\it cyclic dualizing family} \cD\!.  

\item Mirroring Lemma~\ref{girardquantalelem}, if $\cQ$ is Girard, it has a second composition $\op$ given by the linear logic version of the Morgan duality, such that $\cQ^{co}$ is also a quantaloid (where $-^{co}$ denotes the reversal of 2-cells).

\end{itemize}
\end{defn}

\begin{rem}
Quickly, we introduce a bit of notation that will be used throughout this paper. Suppose $(\cV,\ot,\top)$ is a monoidal category then let $\cB(\cV)$ denote its suspension, i.e., the bicategory with one object $\star$ whose 1-cells and 2-cells are the objects and morphisms of \cV, respectively, with composition given by the tensor product $\ot$ and $\rm \top_\star$ given by the unit $\top$ of $\cV$. 
\end{rem}

The next preliminary subsections will outline constructions that give rise to various examples of quantaloids, but we outline here three key examples.

\begin{ex}\cite{Rosenthal_1996}
\begin{enumerate}
\item $\cB(Q)$, the suspension of a unital quantale, is a quantaloid with one object. Note that a Girard quantaloid with one object is a Girard quantale. 

\item \sREL, the locally posetal bicategory of sets and relations, is a quantaloid with hom-sets ordered under inclusion and standard relational composition: if we have $R\colon X\tov Y$ and $S\colon Y\tov Z$ then define \[(x,z)\in R\ot S \quad\mbox{if and only if}\quad \exists y \mbox{\,\,\,\,\,}(x,y)\in R\,\, \mbox{\,and\,} \,\,(y,z)\in S\]

\item \sORD, the locally posetal bicategory of preordered sets and order ideals, is a quantaloid with the standard relational composition: recall that a preordered set is a set $X$ endowed with a reflexive and transitive relation $\leq_X$ and order ideals are relations $R\colon X\tov Y$ such that 
\[ x\leq_X x', \quad x'Ry\quad\implies\quad  xRy \quad{\rm and}\quad  y\leq_X y', \quad xRy\quad\implies\quad  xRy'\]
\end{enumerate}
\end{ex}

\begin{defn}\cite{Rosenthal_1991}
If \cQ\ and $\cQ'$ are quantaloids, then a {\em quantaloid homomorphism} is a functor $F\colon\cQ\rarr\cQ'$ such that on hom-sets it induces a sup-lattice morphism \mbox{$\cQ(a,b)\rarr \cQ'(F(a),F(b))$} for all $a,b\in\cQ$.
\end{defn}

\subsubsection{The category \QRel}

A relation $R:X\tov Y$ assigns a truth value to each pair in $X\times Y$, as such it can be understood as a function from $X\times Y$ to the two-chain quantale. \sREL\ can thus be generalized to arbitrary quantales by considering quantale-valued relations as follows, giving rise to a multitude of quantaloid examples. 

\begin{defn}
If $Q$ is a quantale, we can form the category \QRel\ whose objects are sets and arrows
$R\colon X\tov Y$ are functions $R\colon X\times Y\rarr Q$, called $Q$-relations. Given $R\colon X\tov Y$
and $S\colon Y\tov Z$, the composition $R\ot S:X\tov Z$ is defined by 
\[ R\ot S(x,z)=\bigvee_{y\in Y} R(x,y)\ot S(y,z)\]
Note that the use of $\ot$ on the left refers to composition in \QRel\ and on the right refers to multiplication in $Q$. Identities are given by
\[\top_X(x,x')=\begin{cases} \top& \mbox{if } x=x'  \\ \qzero &\mbox{if } x\neq x'\end{cases}\]
\end{defn}

\begin{lem}\cite{Hofmann_Seal_Tholen_2014}\label{QRel_quantaloid}
If $Q$ is a unital quantale, then \QRel\ is a quantaloid under point-wise ordering, so in particular it is a locally posetal bicategory. 
\end{lem}

As \QRel\ is a quantaloid, given a Q-relation $R\colon X\tov Y$, there exists residuation functors $(-)\lTo R$ and $R\rTo(-)$, defined for $T\colon X\tov Z$ and $U\colon W\tov Y$ by
\begin{align*}
&R\rTo\, T(y,z) = \bigwedge_{x\in X} R(x,y)\rTo\, T(x,z) & U\,\lTo R(w,z) = \bigwedge_{y\in Y} U(w,y)\,\lTo R(x,y)
\end{align*}

\begin{ex}
\begin{enumerate}
	\item $\sREL\cong 2$-\sREL, the quantaloid of sets and relations. 
	\item ${\sf P_{+}}$-\sREL, a quantaloid of sets and extended distance relations $d\colon X\times Y\rarr [0,\infty]$, with composition $\ot$ defined, for $D_1\colon X\times Y\rarr [0,\infty]$ and $D_2\colon Y\times Z\rarr [0,\infty]$,  by \[ (D_1\ot D_2)(x,z) = \bigwedge_{y\in Y} D_1 (x,y)+D_2(y,z) \] and identities $\top_X$ defined by \[\top_X(x,x')=\begin{cases} 0& \mbox{if } x=x'  \\ \infty &\mbox{if } x\neq x'\end{cases}\]
	
	Alternatively, one can consider $[0,1]$-$\sREL$, isomorphic to ${\sf P_{+}}$-\sREL as the map $[0,1]\cong {\sf P_{+}}$ extends to an isomorphism of quantaloids, with composition $\ot$ defined, for $D_1\colon X\times Y\rarr [0,1]$ and $D_2\colon Y\times Z\rarr [0,1]$,  by \[ (D_1\ot D_2)(x,z) = \bigvee_{y\in Y} D_1 (x,y)\cdot D_2(y,z) \] and identities $\top_X$ defined by \[\top_X(x,x')=\begin{cases} 1& \mbox{if } x=x'  \\ 0 &\mbox{if } x\neq x'\end{cases}\]
	
	\item ${\sf P_{max}}$-\sREL, another quantaloid of sets and extended distance relations $d\colon X\times Y\rarr [0,\infty]$, but with a different composition $\ot$ defined, for $D_1\colon X\times Y\rarr [0,\infty]$ and $D_2\colon Y\times Z\rarr [0,\infty]$,  by \[ (D_1\ot D_2)(x,z) = \bigwedge_{y\in Y} {\sf max}(D_1 (x,y),D_2(y,z)) \] 
	
	The last two examples will connect to Lawvere metric spaces. This will be discussed in \mbox{Example \ref{Q-Mod_Ex}.}
\end{enumerate}
\end{ex}

Note that there is a quantaloid embedding (quantaloid homomorphism which is injective on objects and a faithful functor) from the suspension of the quantale $\mathcal{ B}(Q)$ into \QRel:
\[ \mathcal{ B}(Q)\hookrightarrow \QRel: \quad\quad \star\mapsto 1=\{*\} \quad\quad a\mapsto R_a\colon 1\tov 1\quad {\rm where}\quad R_a(*,*)=a \]

\subsection{Modules, matrices and monads}

We assume the reader is familiar with the theory of bicategories, introduced by Benabou. Appropriate references are \cite{Betti_Carboni_Street_Walters_1983, Leinster_1998}. In this section, we introduce three standard constructions in bicategory theory, but first we must review the notion of a biclosed bicategory.

\subsubsection{Biclosed bicategories}
Recall the definitions of right extensions and right liftings, in the sense of \cite{Street_2014}.

\begin{defn}\label{ext}
Let \cB\ be a bicategory. A {\em right extension} of $B\colon X\rarr Z$ along $A\colon X\rarr Y$ in \cB\ is a 1-cell which we denote $A\rTo~B$, also 
denoted by $X{\rm Mod}(A,B)$, together with a 2-cell 
$$\bfig
\ptriangle/->`->`-->/<400,400>[X`Y`Z;A`B`A\rTo B]
\morphism(220,250)<-180,0>[`;\varepsilon_{X,A}]
\efig$$
inducing a bijection (called the {\em transpose}) between 2-cells $C\rarr A\,\rTo B$ and $A\ot C \rarr B$, for all $C\colon Y\rarr Z$. 

\noindent A {\it right lifting} of $C\colon Z\rarr Y$  through $A\colon X\rarr Y$ is a right extension in $\mathcal{ B}^{op}$, i.e., a 1-cell which we denote $C~\lTo A$, also denoted by ${\rm Mod}Y(A,C)$, together with a 2-cell 
$$\bfig
\dtriangle/<--`->`->/<400,400>[X`Z`Y;C\lTo A`A`C]
\morphism(340,240)|l|<0,-170>[`;\varepsilon_{C,Y}]
\efig$$
inducing a bijection (called the {\em transpose}) between 2-cells $B\rarr C\lTo A$ and $B\ot A \rarr C$, for all $B\colon Z\rarr X$. \\

\noindent A bicategory \cB\ is called {\it biclosed} if it admits all right extensions and right liftings.
\end{defn}

\begin{rem} \label{closed} 
We have chosen to use the more traditional notation and terminology from bicategory theory. \cite{Cockett_Koslowski_Seely_2000} chose their terminology to be in agreement with Lambek's non-commutative linear logic rather than the theory of bicategories. We note that our notion of right extension is the same as their notion of {\it right hom} and our notion of right lifting is their {\it left hom}.
\end{rem}

\begin{ex}
Suppose $(\cV,\ot,\top)$ is a symmetric monoidal closed category and consider its suspension $\cB(\cV)$. Taking $A \rTo~B = {\rm Hom}(A,B)$ and $C \lTo A = {\rm Hom}(A,C)$, we get: $\cB(\cV)$ is a biclosed bicategory.
\end{ex}

\begin{prop}\label{natural} If \cB\ is a biclosed bicategory, then the induced 2-cells $$(A\ot C) \rTo\, B \rarr C\rTo\, (A\rTo\, B)\quad\quad {\rm and}\quad\quad D\,\lTo(A\ot C)\rarr (D\,\lTo C)\,\lTo A$$ are invertible and natural in $A\colon X\rarr Y$,  $B\colon X\rarr W$, $C\colon Y\rarr Z$ and $D\colon W\rarr Z$.
\end{prop}
\begin{proof} This follows from the universal properties of the extensions and liftings.
\end{proof}

\subsubsection{\cB-categories and modules}

The following definition was first introduced for locally posetal bicategories by \cite{Walters_1981} and then defined for more general bicategories by \cite{Street_2005}.

\begin{defn}
Let \cB\ be a biclosed bicategory.
\begin{itemize}
\item A {\em \cB-category} $M$ consists of the following data:
\begin{itemize}
	\item for each 0-cell $A\in\cB$, a set $M_A$ ``over $A$'',
	\item for each pair of elements $x, x'$ over 0-cells $A, B$ respectively, a 1-cell $M(x,x')\colon A\rarr B$ in \cB,
	\item for each triple of elements $x, x', x''$ over 0-cells $A, B, C$ respectively, 2-cells \mbox{$\eta\colon1_A\rarr M(x,x)$} and $\mu\colon M(x,x')\ot M(x',x'')\rarr M(x,x'')$ in \cB\
\end{itemize}
satisfying the axioms of left and right identities, and associativity.

Alternatively, a \cB-category is lax functor from the union of sets $M_A$ viewed as an indiscrete bicategory to \cB.

\item A {\em \cB-module} $\Theta\colon M\tov N$ assigns:
\begin{itemize}
	\item to each pair $x\in M_A$ and $y\in N_B$ over 0-cells $A$ and $B$, a 1-cell $\Theta(x,y)\colon A\rarr B$ in \cB,
	\item to each triple $x, x'\in M_A$ and $y\in N_B$ over 0-cells $A$ and $B$, a left action 2-cell $\rho\colon M(x,x')\ot \Theta(x',y)\rarr \Theta(x,y)$ in \cB, 
	\item to each triple $x\in M_A$ and $y, y'\in N_B$ over 0-cells $A$ and $B$, a right action 2-cell $\lambda\colon \Theta(x,y)\ot N(y,y')\rarr \Theta(x,y')$ in \cB, 
\end{itemize}
satisfying five compatibility axioms.

\item Given \cB-modules $\Theta, \Phi\colon M\tov N$, a morphism of \cB-modules $\alpha\colon \Theta\rarr\Phi$ is a family of 2-cells $\alpha\colon\Theta(x,y)\rarr\Phi(x,y)$ in \cB\ compatible with the left and right actions $\lambda, \phi$. 

\item Given \cB-modules $\Theta\colon M\tov N$ and $\Pi:N\tov P$, the composite $\Theta\ot\Pi\colon M\tov N$ is defined, for $x\in M_A$ and $z\in P_C$ over $A$ and $C$, by $\Theta\ot\Pi(x,z)$ as a colimit in $\cB(A, C)$.

\item By restricting to certain bicategories \cB, we can define the bicategory \BMod\ of \cB-categories and \cB-modules. For more details, see \cite{Street_2005}.
\end{itemize}
\end{defn}

Suppose $(\cV,\ot,\top)$ is a complete and cocomplete symmetric monoidal closed category.  Then we can consider \cB(\cV)-\Mod$\,\cong\,$\cV-\prof, the bicategory of \cV-categories, \cV-profunctors (sometimes called \cV-distributors), and \cV-transformations. \\

\noindent Note that composition $A\ot_RB$ of \cV-profunctors $A\colon Q\tov R$ and $B\colon R \tov S$ can be described by the coend $$(A\otimes_R B)(q,s)=\int^r A(q,r)\otimes B(r,s)$$ and $Q( -, -)\colon Q\tov Q$ is the identity 1-cell.

\subsubsection{\cB-matrices}

The work of categories enriched in bicategories was further developed by \cite{Betti_Carboni_Street_Walters_1983}. They introduced the bicategory of \cB-matrices as a stepping stone to the study of \BMod.

\begin{defn}
Let \cB\ be a locally small-cocomplete bicategory with a small set of objects $\cB_0$. 
\begin{itemize}
	\item Given a family $X=(X_A)_{A\in \mathcal{ B}_0}$ of small sets indexed by $\cB_0$, an element $x\in X_A$ is said to be an {\em element of $X$ over $A$.}
    
	\item Given a pair of families $X=(X_A)_{A\in \mathcal{B}_0}$ and $Y=(Y_A)_{A\in \mathcal{B}_0}$, a {\em \cB-matrix} $S\colon X\tov Y$ assigns to each pair $x,y$ of elements over $A, B\in\cB_0$ a 1-cell $S(x,y)\colon A\rarr B$ in \cB. Composition of \cB-matrices $S\colon X\tov Y$ and $T\colon Y\tov Z$ is by matrix multiplication, i.e., $(S\ot T)(x,z) = \coproduct_{y\in Y} S(x,y)\ot T(y,z)$.
	\item A {\em morphism of \cB-matrices} $\alpha\colon S\rarr S'$ is a family of 2-cells $\alpha_{x,y}\colon S(x,y)\rarr S'(x,y)$ in \cB.
	\item Define \MatrB\ to be the bicategory with families of small sets indexed by $\cB_0$ as 0-cells, \cB-matrices as 1-cells and \cB-matrix morphisms as 2-cells. 
\end{itemize}
\end{defn}

Suppose $(\cV,\ot,\top)$ is a symmetric monoidal closed category with set-indexed products and coproducts. Then, we can consider \Matr\cB(\cV)$\,\cong\,$\cV-\Matr\ the bicategory of sets, $\cV$-matrices, and  $\cV$-matrix morphisms. \\

\noindent Recall from \cite{Carboni_Kasangian_Walters_1987} that a $\cV$-matrix $A\colon X\tov Y$  is a function $A\colon X\times Y\rarr {\sf ob}\cV$. A morphism $f\colon A\rarr B$ of $\cV$-matrices $X\tov Y$ is a family \[\{f (x,y) \colon A(x,y)\rarr B(x,y) \mid (x,y)\in X\times Y\}\] of morphisms of $\cV$.  Composition of $A\colon X\tov Y$ and $B\colon Y\tov Z$ is given by matrix multiplication \[(A\otimes_Y B)(x,z)=\coproduct_{y\in Y} A(x,y)\otimes B(y,z)\] and the identity $X\tov X$ is

 $$\top_X(x,x')=
  \begin{cases}
   \top& x=x' \\
    \qzero & \text{otherwise}\\
\end{cases}$$

\noindent where $\top$ is the unit for $\ot$ and  $\qzero$ is initial in $\cV$. \\

Given \cV-matrices $A\colon X\tov Y$, $B\colon X\tov Z$, and $C\colon Z \tov Y$, taking
\[ A\rTo\,B(y,z)= \prod_{x\in X}  \cV(A(x,y),B(x,z)) \qquad C\,\lTo A(z,x)= \prod_{y\in Y}  \cV(A(x,y),C(z,y)) \] we get the well-known result within folklore, see \cite{Blute_Cockett_Seely_Trimble_1996} and \cite{Lack_2010}:
\begin{lem}\label{VMat_biclosed} If \cV\ is a symmetric monoidal closed category with set-indexed products and coproducts, then \cV-\Matr\ is a biclosed bicategory.
\end{lem}

\subsubsection{Monads in \cB\ and modules}

\cite{Carboni_Kasangian_Walters_1987} demonstrated that \BMod\ can be broken into the constructions of \MatrB\ and of \MonB, the bicategory of monads and modules in \cB\ as defined below.

\begin{defn}
Let \cB\ be a biclosed bicategory with local equalizers and coequalizers stable under composition. 
\begin{itemize}
	\item A {\em monad} in \cB\ is a 1-cell $Q\colon X\rarr X$ together with two 2-cells $e\colon 1_X\rarr Q$ and 
$m\colon Q\ot Q \rarr Q$ satisfying the usual associativity and identity axioms. Note that a monad $(X,Q)$ (denoted 
by just $Q$ when $X$ is understood)  is  a monoid in the category $\mathcal{ B}(X,X)$. 
	\item A {\em $(Q,R)$-module} $A\colon(X,Q)\tov(Y,R)$ is a 1-cell $A\colon X\rarr Y$ together with action 2-cells $\lambda\colon Q\ot A\rarr A$ and $\rho\colon A\ot R\rarr A$ satisfying (one-sided) associative and unit laws, and a diagram expressing the commutativity of the two actions.
	\item A {\em $(Q,R)$-module morphism} is a 2-cell $f\colon A\rarr B$ satisfying compatibility conditions with respect to the actions.
\item Let \MonB\ denote the bicategory of monads, modules, and module morphisms; with the composition $A\ot_R B$ of $A\colon (X,Q) \tov (Y,R)$ and $B\colon (Y,R)\tov (Z,S)$ given by the local coequalizer \[ A\ot R \ot B\two ^{\rho_A\ot B}_{A\ot \lambda_B}A \ot B\to A\ot_RB\] in  $\mathcal{ B}(X,Z)$, and  identity 1-cell $\ttop_{(X,Q)}=Q\colon (X,Q)\tov (X,Q)$.
\end{itemize}

\end{defn}

Note that identity 1-cell $\top_X\colon X\rarr X$ in \cB\ is a trivial monad, every 1-cell $A\colon X\rarr Y$ is an $(\top_X,\top_Y)$-module, and every 2-cell $f\colon A\rarr B$ is a $(\top_X,\top_Y)$-module homomorphism. Moreover, $X\mapsto \top_X$ defines a pseudo functor which is a left pseudo adjoint to the forgetful (lax) functor $\MonB\rarr\cB$. \\

Given $A\colon (X,Q)\tov (Y,R)$, $B\colon (X,Q)\tov (Z,S)$, and $C\colon (Z,S) \tov (Y,R)$, the right extension $B\rTo_{Q}\, A$ in \MonB\ of $B$ along $A$ is obtained by taking the local equalizer.
\[ A\rTo_{Q}\, B\to<125> A\rTo\, B  \two ^{\lambda^*_A}_{\hat \lambda_B} (Q\ot A)\rTo\,B \] 
where $\lambda^*_A = \lambda_A \rTo\, B$ and $\hat \lambda_B$ is the transpose of 
\[ (Q\ot A)\ot (A\rTo\, B)\xrightarrow{\sim} Q\ot (A\ot (A\rTo\, B)) \xrightarrow{Q\ot \epsilon_{X,A}} Q\ot B \xrightarrow{\lambda_B} B \]
\noindent Similarly the right lifting $C\,\lTo_R A$ in \MonB\ of $C$ through $A$ is 
\[ C\,\lTo_R A\to<125> C\,\lTo A \two ^{\rho^*_A}_{\hat \rho_B}  C\,\lTo (A\ot R) \] 

\noindent Note that given morphisms $A\colon X\rarr Y, B\colon X\rarr Z$ and $C\colon Z\rarr Y$ in \cB\, the 2-cells 
\[ A\rTo_{\top_X}\, B\rarr A\rTo\, B \quad\quad{\rm and}\quad\quad C\,\lTo_{\top_Y} A\rarr C\,\lTo A  \] 
defining the extensions and liftings in \MonB, are invertible, and without loss of generality, we can take them to be equalities. \\

Moreover, given the above definitions, the following is a well-known result, see \cite{Betti_Carboni_Street_Walters_1983} and \cite{Lack_2010}:
\begin{lem}\label{MonB_biclosed} 
If \cB\ is a biclosed bicategory with local equalizers and coequalizers stable under composition, then \MonB\ is a biclosed bicategory.
\end{lem}

By examining the definitions of the three constructions, it is immediate that: 
\begin{prop}\cite{Carboni_Kasangian_Walters_1987}\label{BMod_as_MonMatrB}
If \cB\ is a distributive bicategory (locally cocomplete bicategory with colimits preserved by composition on both sides), \BMod\ is biequivalent to \Mon\MatrB.
\end{prop}

\noindent Furthermore, the constructions are idempotent.
\begin{prop}\cite{Carboni_Kasangian_Walters_1987}\label{idempotent}
If \cB\ is a distributive bicategory, then (\BMod)-\Mod\ is biequivalent to \BMod. If \cB\ is a bicategory with local coequalizers stable under composition, then \Mon(\MonB) is biequivalent to \MonB. If \cB\ is a bicategory with small local coproducts stable under composition, then \Matr(\MatrB) is biequivalent to \MatrB. 
\end{prop}

Thus, by Lemmas \ref{VMat_biclosed} and \ref{MonB_biclosed}, and calculating the liftings and extensions by the ends
\[(A \rTo B)(r,s)=\int_q {\rm Hom}(A(q,r),B(q,s)) \qquad (C \lTo A)(s,q)=\int_r {\rm Hom}(A(q,r),B(s,r)) \] for \cV-profunctors $A\colon Q\tov R$, $B\colon Q\tov S$ and $C\colon S\tov R$, we get:

\begin{lem}\label{VProf_biclosed}
Suppose \cV\ is a complete and cocomplete symmetric monoidal closed category, then \cV-\prof\ is a biclosed bicategory.
\end{lem}

\subsubsection{Enrichment in a quantaloid \cQ}

The previous bicategorical constructions can of course be considered in the special case of a quantaloid \cQ, as done by \cite{Rosenthal_1992}. We take the time to describe them in detail as a main result of this paper is their generalization to the linear context.

\begin{defn}\label{QMod}Let \cQ\ be a quantaloid.
\begin{itemize}

\item A {\em \cQ-category} is a pair $M=(X,\rho)$ where:
\begin{itemize}
    \item $X$ is a set,
    \item $\rho$ is a function $\rho\colon X\rarr {\sf ob}\cQ$, and 
    \item there is a function, called the {\em enrichment} assigning to each pair $(x,x')\in X\times X$ a morphism $M(x,x')\colon\rho(x)\rarr\rho(x')$ such that $\forall x, x', x''\in X$
    \[ \top_{\rho(x)}\leq M(x,x) \qquad M(x,x')\ot M(x',x'')\leq M(x,x'')\]
\end{itemize}
Alternatively, a \cQ-category is a lax functor from set $X$ to \cQ, when $X$ is viewed as an indiscrete bicategory.

\item Consider \cQ-categories $M=(X,\rho_M)$ and $N=(Y,\rho_N)$. A {\em \cQ-module} $\Theta\colon M\tov N$ consists of an assignment of a morphism $\Theta(x,y)\colon \rho_M(x)\rarr\rho_N(y)$ to every pair $(x,y)\in X\times Y$ such that $\forall x, x'\in X, y, y'\in Y$
\[ \Theta(x,y) \ot N(y,y')\leq \Theta(x,y') \qquad M(x,x')\ot\Theta(x',y) \leq \Theta(x,y) \]

\item Define the category \QMod\ whose objects are \cQ-categories and arrows are \cQ-modules. Given $M\stackrel{\Theta}{\tov}N \stackrel{\Pi}{\tov}P$, composition $\Theta\ot\Pi\colon M\tov P$ is defined by
\[ (\Theta\ot\Pi)(x,z)=\bigvee_{y\in Y}\Theta(x,y)\ot\Pi(y,z) \]
Identity 1-cells $\ttop_M\colon M\tov M$ are defined by $\ttop_M(x,x')=M(x,x')$. Note that the use of $\ot$ on the left refers to composition in \QMod\ and on the right refers to composition in \cQ. 
\end{itemize}
\end{defn}

\begin{rem}
We have chosen to denote the bicategory of \cQ-categories and \cQ-modules as \QMod\ in agreement with previously introduced notation, while it is called {\sf Bim}(\cQ) by \cite{Rosenthal_1992}.
\end{rem}

\begin{defn}
Let \cQ\ be a quantaloid. Define the category \MatrQ\ whose objects are small families of objects in \cQ, i.e., pairs $(X, \gamma)$ of a set $X$ and a function $\gamma\colon X\rarr ob\cQ$, and arrows are \cQ-matrices $r\colon(X, \gamma)\tov (Y,\phi)$, i.e., families of morphisms $(r_{x,y}\colon\gamma(x)\rarr\phi(y))_{(x,y)\in X\times Y}$ in \cQ. Given \cQ-matrices $(X, \gamma)\stackrel{r}{\tov}(Y,\phi)\stackrel{s}{\tov}(Z,\chi)$, composition $r\ot s\colon(X,\gamma)\tov (Z,\chi)$ is defined by 
\[ (r\ot s)_{x,z}=\bigvee_{y\in Y}r_{x,y}\ot s_{y,z} \]
Identity 1-cells $\ttop_{(X,\gamma)}\colon(X,\gamma)\tov (X, \gamma)$ are defined by ${\ttop_{(X,\gamma)}}_{x,x'} = \begin{cases}
\top_{\gamma(x)}& \mbox{if } x=x' \\  \qzero_{\gamma(x),\phi(x')} &\mbox{if } x\neq x' \end{cases}$.
\end{defn}

\begin{ex}
\begin{enumerate}
	\item It is a standard result that the category of $\cB(\Omega)$-matrices, $\Matr(\cB(\Omega))$, is isomorphic to $\sREL$.
	\item Consider $\Matr(\cB(\Omega))$. The objects are pairs $(X,\gamma)$ of a set $X$ and a trivial function mapping the set to $\star$. The $\cB(Q)$-matrices are families $(r_{x,y}\colon\star\rarr\star)_{(x,y)\in X\times Y}$, a collections of elements in $Q$, i.e., a function $X\times Y\rarr Q$, and composition of $\cB(Q)$-matrices is exactly the one of $Q$-relations. As such, {\sf Matr}$\Matr(\cB(Q))\,\cong\,\QRel$, generalizing the first example.
\end{enumerate}
\end{ex}

\begin{defn}
Let \cQ\ be a quantaloid. 
\begin{itemize}
	\item A {\em monad} $(a,m)$ in \cQ\ is an object $a$ equipped with a morphism $m\colon a\rarr a$ such that 
	\[\top_a\leq m \qquad m \ot m \leq m\]
	Alternatively, a monad in \cQ\ is a lax functor from the terminal bicategory 1 to \cQ.
	
	\item A {\em $(m,n)$-module} $f\colon(a,m)\tov(b,n)$ is a morphism $f\colon a\rarr b$ in \cQ\ such that 
	\[ m\ot f\leq f \qquad f\ot n\leq f\]
	
	\item \MonQ\ is the category of monads and monad modules. in \cQ. Composition is directly inherited from \cQ\ and the identity 1-cell $\ttop_{a,m}\colon (a,m)\tov(a,m)$ is $m:a\rarr a$.
\end{itemize}
\end{defn}

\begin{ex}
\begin{enumerate}
	\item It is immediate that monads in \sREL\ are preordered sets and monad modules are order ideals, thus $\Mon(\sREL)\,\cong\,\sORD$.
	
	\item The above example can be generalized to \QRel. Consider $\Mon(\QRel)$: monads in \QRel\ are $Q$-categories, pairs $(X, m)$, set $X$ equipped with a relation $m\colon X\tov X$ which is $Q$-reflexive and $Q$-transitive, meaning \[ \top \leq m(x,x) \quad\quad{\rm and}\quad\quad m(x,x')\ot m(x',x'') \leq m(x,x'') \quad\quad \forall x,x',x''\in X \]
	while monad modules are $Q$-profunctors, relations $S\colon (X,m)\tov (Y, n)$ \[ m(x,x')\ot S(x',y) \leq S(x,y) \quad{\rm and}\quad S(x,y')\ot n(y',y) \leq S(x,y) \quad \forall x,x'\in X, y,y'\in Y \]
\end{enumerate}
\end{ex}

These constructions are more than just categories:
\begin{lem}\cite{Rosenthal_1992}\label{QMod_MatrQ_MonQ_Quantaloids}
\QMod, \MatrQ\ and \MonQ\ are quantaloids, so in particular they are locally posetal bicategories. 
\end{lem}

As \QMod, \MatrQ\ and \MonQ\ are quantaloids, there exist residuation functors. In the case of \QMod\ and \MatrQ, the formulas are similar to \QRel, the inf of the point-wise residuations, while \MonQ\ inherits residuations directly from \cQ. \\

Propositions \ref{BMod_as_MonMatrB} and \ref{idempotent} apply to the case of quantaloids:
\begin{prop}\cite{Rosenthal_1992}\label{QMod_as _MonMatrQ}
\QMod\ is biequivalent to ${\sf Mon}\MatrQ$ and the constructions are idempotent.
\end{prop}

\noindent Thus, we see that $\sORD\cong\cB(\Omega)$-$\Mod\cong\sREL$-\Mod\ and $Q$-$\prof\cong\cB(Q)$-$\Mod\cong (\QRel)$-\Mod.

\begin{ex}
\begin{enumerate}\label{Q-Mod_Ex}
	\item Consider $Q={\sf P_{+}}$, then ${\sf P_{+}}$-$\prof$ is the quantaloid of Lawvere metric spaces and ${\sf P_{+}}$-bimodules (in the sense of \cite{Lawvere_1973}). Objects are sets $X$ equipped with extended distance functions $m\colon X\times X\rarr [0,\infty]$ which satisfy the axiom of point inequality and the triangle inequality: \[ m(x,x)\leq 0 \quad\quad{\rm and}\quad\quad m(x,x'')\leq m(x,x')+m(x',x'')\quad\quad \forall x,x',x''\in X \] In other words the objects are extended quasi-pseudo-metric spaces, or simply Lawvere metric spaces. The arrows $(X,m)\tov (Y,n)$ are real-valued functions $F\colon X\times Y\rarr [0,\infty]$ satisfying \[ F(x,y)\leq m(x,x')+F(x',y) \quad{\rm and}\quad F(x,y)\leq F(x,y')+n(y',y) \quad \forall x,x'\in X, y,y'\in Y \] The standard category {\sf Met} of Lawvere metric spaces and non-expansive maps is equivalent to the category of ${\sf P_{+}}$-categories and ${\sf P_{+}}$-functors. Every ${\sf P_{+}}$-functor gives rise to a pair of adjoint  ${\sf P_{+}}$-bimodules, therefore ${\sf P_{+}}$-$\prof$ is a quantaloid of Lawvere metric spaces with a more general notion of morphisms. \\
	
	Alternatively, take $Q=[0,1]$, then $[0,1]$-$\prof\cong{\sf P_{+}}$-$\prof$. Objects are sets $X$ equipped with a functions $m\colon X\times X\rarr [0,1]$ satisfying
	\[ 1\leq m(x,x) \quad\quad{\rm and}\quad\quad m(x,x')\cdot m(x',x'')\leq m(x,x'') \quad\quad \forall x,x',x''\in X \] while arrows $(X,m)\tov (Y,n)$ are functions $F\colon X\times Y\rarr [0,1]$ satisfying \[ m(x,x')\cdot F(x',y)\leq F(x,y) \quad{\rm and}\quad F(x,y')\cdot n(y',y)\leq f(x,y) \quad \forall x,x'\in X, y,y'\in Y \]
	
	\item If instead, we consider $Q={\sf P_{max}}$, then ${\sf P_{max}}$-$\prof$ is the quantaloid of Lawvere ultrametric spaces and ${\sf P_{max}}$-bimodules. The metric spaces $(X,m)$ satisfy the strengthened triangle inequality: \[m(x,x'')\leq {\sf max}(m(x,x'), m(x',x''))\quad\quad \forall x,x',x''\in X \] and the arrows $(X,m)\tov (Y,n)$ are real-valued functions $F\colon X\times Y\rarr [0,\infty]$ satisfying \[ F(x,y)\leq {\sf max}(m(x,x'), F(x',y)) \quad{\rm and}\quad F(x,y)\leq {\sf max}(F(x,y'),n(y',y)) \quad \forall x,x'\in X, y,y'\in Y \]

For more details, see \cite{Lawvere_1973}. 


\end{enumerate}
\end{ex}

Note that there is a quantaloid embedding of \cQ\ into \QMod,
\[\cQ\hookrightarrow\QMod: \quad\quad a\mapsto M_a = (1, \rho_a) \quad{\rm where}\quad 1={*},\quad \rho_a(*)=a\quad {\rm and}\quad M_a(*,*)=\top_a \]
\[ f\colon a \rarr b \mapsto \Theta_f\colon M_a\tov M_b\quad{\rm where}\quad \Theta_f(*,*)=f \]

of \cQ\ into \MatrQ,
\[\cQ\hookrightarrow\MatrQ: \quad\quad a\mapsto (1, \gamma_a) \quad{\rm where}\quad 1={*},\quad \gamma_a(*)=a\quad \]
\[ \quad\quad\quad\quad\quad\quad\quad\quad\quad\quad\quad\quad f\colon a \rarr b \mapsto f\colon (1,\gamma_a)\tov (1,\gamma_b)\quad{\rm where}\quad f_{*,*}=f \]

and of \cQ\ into \MonQ,
\[\cQ\hookrightarrow\MonQ: \quad\quad a\mapsto (a,\top_a)~~{\rm the~trivial~monad} \]
\[ \quad\quad\quad\quad\quad\quad\quad\quad\quad\quad f\colon a \rarr b \mapsto f\colon (a,\top_a)\tov (b,\top_b) \]

\section{Linear bicategories}

We introduce the theory of {\it linear bicategories} as defined by \cite{Cockett_Koslowski_Seely_2000}.  The material in this subsection is entirely from that paper.\\

Linear bicategories are an extension of the theory of {\em linearly distributive categories}, due to Cockett and Seely \cite{Cockett_Seely_1997}. Linearly distributive categories axiomatize the multiplicative fragment of linear logic in a way that is closer to the syntax. So the two binary connectives, $\otimes$\ and $\oplus$, are taken as primitives, and negation can be added if one wishes. \\

As usual with bicategories, one begins with a class of {\it 0-cells} which we will denote \mbox{$\cB_0=\{X,Y,Z,\ldots\}$}. Then for every pair of 0-cells, one has a category $\cB(X,Y)$. The objects of $\cB(X,Y)$ are called {\it 1-cells} and the arrows are called {\it 2-cells}. But now we have two composition functors:
\[\ot,\op\colon \cB(X,Y)\times\cB(Y,Z)\lrarr\cB(X,Z)\]
Each of these compositions gives a bicategory structure. Thus for each composition we have all of the morphisms and coherences that this entails. In particular we must have identity 1-cells for each of the two compositions:
\[\top_X\colon X\rarr X \qquad\mbox{and}\qquad  \bot_X\colon X\rarr X\]

These two bicategory structures are related by a linear distribution as follows. Given 0-cells $X, Y, Z, W$ we have two functors:
\[ -\ot(-\op-), (-\ot-)\op-\colon \cB(X,Y)\times\cB(Y,Z)\times\cB(Z,W)\lrarr \cB(X,W)\]
and we require a natural transformation between them, which is not necessarily an 
isomorphism: 
\[\delta^L:-\ot(-\op -)\Longrightarrow (-\ot-)\op-\]
Symmetrically, we also need a natural transformation 
\[\delta^R:(-\op-)\ot-\lrarr-\op(-\ot-).\]
All of this structure must satisfy coherence requirements detailed in \cite{Cockett_Koslowski_Seely_2000}. \\

One of the main goals of this paper will be to provide new examples of linear bicategories, but we mention here the quintessential example \sREL, as it will be the example on which ours are based.\\

\sREL\ is a locally posetal bicategory with its first composition being the standard one.  But we have a second composition:  for $R\colon X\tov Y$ and $S\colon Y\tov Z$, define
\[(x,z)\in R\op S \mbox{\,\,\,\,\,\, if and only if\,\,\,\,} \forall y \mbox{\,\,\,\,\,\,}(x,y)\in R\,\, 
\mbox{\, or\,} \,\,(y,z)\in S\]

We quickly mention here that the appropriate notion of functor between linear bicategories was developed, although they will only make a brief appearance in this article.

\begin{defn}\cite{Cockett_Koslowski_Seely_2000}
Let \cB\ and \cB' be linearly distributive categories. A {\em linear functor} $F = (F_\ot, F_\op)\colon\cB\rarr\cB'$ consists of:
\begin{itemize}
\item a lax monoidal functor $(F_\ot, m_\top, m_\ot)\colon(\cB,\ot,\top)\rarr(\cB',\ot,\top)$, equipped with
\item a lax comonoidal functor $(F_\op, n_\bot, n_\op)\colon(\cB,\op,\bot)\rarr(\cB',\op,\bot)$, equipped with
\item four natural transformations, known as {\em linear strengths}, 
\begin{align*}
&v_\ot^R\colon \op; F_\ot \rarr (F_\op \times F_\ot); \op & {v_\ot^R}_{A,B} \colon F_\ot(A\op B)\rarr F_\op(A)\op F_\ot(B)\\
&v_\ot^L\colon \op; F_\ot \rarr (F_\ot \times F_\op); \op & {v_\ot^L}_{A,B} \colon F_\ot(A\op B)\rarr F_\ot(A)\op F_\op(B)\\
&v_\op^R\colon (F_\ot\times F_\op);\ot \rarr \ot; F_\op & {v_\op^R}_{A,B} \colon F_\ot(A)\ot F_\op(B)\rarr F_\op(A\ot B)\\
&v_\op^L\colon (F_\op\times F_\ot);\ot \rarr \ot; F_\op & {v_\op^L}_{A,B} \colon F_\op(A)\ot F_\ot(B)\rarr F_\op(A\ot B)
\end{align*}
\end{itemize}
subject to various coherence conditions.
\end{defn}

\subsubsection{Cyclic $*$-autonomous bicategories}

The notion of a $*$-autonomous category was originally introduced independently of linear logic by \cite{Barr_1979}, who was trying to capture some of the dualities present in various categories of topological vector spaces. It was only later seen to be appropriate to model the multiplicative fragment of linear logic by \cite{Seely_1989}. Barr's definition of $*$-autonomous category was a symmetric monoidal closed category with a dualizing object. But linear logic suggested an equivalent definition as a linearly distributive category with negation, see \cite{Cockett_Seely_1997}. Similarly, as linearly distributive categories have been generalized to linear bicategories, it is natural to generalize the notion of cyclic $*$-autonomous categories as well.

\begin{defn}\cite{Koslowski_2001}\label{cyclic_starauto}
A bicategory \cB\ is a {\em cyclic $*$-autonomous bicategory} if 
\begin{itemize}
	\item for any pair of 0-cells $X, Y$, there is an adjoint equivalence \mbox{$(-)^*\dashv ((-)^*)^{op}:\cB(X,Y)\rarr\cB(Y,X)^{op}$}, and
	\item for any 1-cell $A:X\rarr Y$, the 1-cell $A^*$ is the right extension of $\top_X^*$ along $A$, such that these extensions are natural in $A$. 
\end{itemize}
\end{defn}

As discussed by \cite{Cockett_Koslowski_Seely_2000} and \cite{Koslowski_2001}, we can equivalently define cyclic $*$-autonomous bicategory by introducing the concept of dualizing 1-cells.

\begin{defn}
Suppose $\cD=\{\bot_X\colon X\rarr X \, \vert \, X \in \cB\}$ is a family of 1-cells in a biclosed bicategory \cB. Given $A\colon X\rarr Y$, applying Definition~\ref{ext}, we get 2-cells 
\[A\,  \to<220>^{\delta_{A,Y}} \, (\bot_Y \lTo A)\rTo \bot_Y \quad\quad {\rm and} \quad \quad A \, 
\to<220>^{\delta_{X,A}} \, \bot_X \lTo (A \rTo \bot_X) \]
\begin{itemize}

\item A family $\mathcal{ D}$ is called {\em dualizing} if the 2-cells $\delta_{X,A}$ and $\delta_{A,Y}$ are invertible, for all $A\colon X\rarr Y$.

\item A dualizing family  $\mathcal{ D}$ is called {\em cyclic} if there are invertible 2-cells $\theta_A\colon \bot_Y \lTo A \xrightarrow{\sim} A \rTo \bot_X$, natural in $A\colon X\rarr Y$, such that  the following diagram commutes
\[\bfig
\Ctriangle/<-``->/<600,300>[\qquad \qquad \bot_X \lTo  (A  \rTo \bot_X)`A`\qquad\qquad(\bot_Y \lTo A) \rTo \bot_Y;\delta_{X,A}``\delta_{A,Y}]
\morphism(600,600)|r|<0,-300>[\qquad \qquad \bot_X \lTo (A \rTo \bot_X)`\qquad \qquad \bot_X \lTo (\bot_Y \lTo  A);\bot_X  \lTo\theta_A]
\morphism(600,300)|r|<0,-300>[\qquad \qquad \bot_X \lTo (\bot_Y \lTo  A)`\qquad\qquad(\bot_Y \lTo A) \rTo \bot_Y;\theta_{\bot_Y \lTo A}] 
\efig\]
In this case, we let $A^\perp=A \rTo \bot_X$.

\end{itemize}
\end{defn}

\begin{prop}
\cB\ is a cyclic $*$-autonomous bicategory if and only if it admits a cyclic dualizing family.
\end{prop}

\begin{ex}\label{Girard_Star} Suppose \cV\ is a $*$-autonomous category with cyclic dualizing object $\bot$, one can show that  $\cD=\{ \perp\colon\star \to<125> \star \}$ is a cyclic dualizing family for its suspension $\cB(V)$, since $A\rTo \bot \cong A^\perp\cong \bot\lTo A$ and $A\cong (A^\perp)^\perp$, for all objects $A\in\cV$. So $\cB(V)$ is a cyclic $*$-autonomous bicategory. \\

\noindent In particular, consider $\cV={\rm Sup}$, the category of suplattices and suplattice homomorphisms (functions that preserve all joins). {\rm Sup} is a $*$-autonomous category with cyclic dualizing object $\Omega^{op}$, with $A\,\lTo\Omega^{op}\cong \Omega^{op}\rTo\, A\cong A^\circ$, where $A^\circ$ denotes the opposite poset of $A$, and $A\cong (A^\circ)^\circ$, for all $A$. Therefore, $\cB({\rm Sup})$ is a  cyclic $*$-autonomous bicategory with cyclic dualizing family $\cD=\{\Omega^{op}\colon\star \rarr \star \}$. 
\end{ex}

It will come as no surprise then that every cyclic $*$-autonomous bicategory is a linear bicategory by the de Morgan equations. It is remarked in \cite{Cockett_Koslowski_Seely_2000} without a detailed proof as it would largely follow the same computation described in \cite{Cockett_Seely_1997}, in the case of LDCs and $*$-autonomous categories. We describe below some of the key details.

\begin{lem}\label{lemma} If \cB\ is a cyclic $*$-autonomous bicategory, then there is an invertible 2-cell $$A^\perp\rTo\,B\xrightarrow{\sim} (B^\perp \ot A^\perp)^\perp \quad\quad {\rm and}\quad\quad B\,\lTo C^\perp\xrightarrow{\sim} (C^\perp \ot B^\perp)^\perp$$ for all $A\colon X\rarr Y$, $B\colon Y\rarr Z$ and $C\colon Z\rarr W$.
\end{lem}

\begin{proof} Suppose $\mathcal{ D}=\{{\bot_X}\colon X\rarr X \}$  is a cyclic dualizing family for \cB. Then composition with $\delta_{B,Z}$ induces the transpose invertible 2-cell 
\[ A^\perp\rTo\,B \xrightarrow{\sim} A^\perp\rTo\,((\bot_Z\,\lTo B)\rTo\,\bot_Z) \]
Since $\bot_Z\,\lTo B \cong  B ^\perp$ and by Proposition \ref{natural}, it follows that $$A^\perp\rTo\,((\bot_Z\,\lTo B)\rTo\,\bot_Z)\xrightarrow{\sim} (B^\perp\ot A^\perp)\rTo\,\bot_Z = (B^\perp\ot A^\perp)^\perp$$ and the desired 2-cell follows. Similarly for the second result.
\end{proof}

While the following result follows from the discussion by \cite{Cockett_Koslowski_Seely_2000} in Section 3.4, we outline the broad strokes of the proof for readers less initiated in categorical linear logic. 

\begin{prop}\cite{Cockett_Koslowski_Seely_2000}\label{Girard} Every cyclic $*$-autonomous bicategory is a linear bicategory.
\end{prop} 

\begin{proof}
Suppose $\cD=\{{\bot_X}\colon X\rarr X \}$ is a cyclic dualizing family for \cB. Given  $A\colon X\rarr Y$ and $B\colon Y\rarr Z$, define $A\op B= (B^\perp\otimes A^\perp)^\perp$. \\

Then $\op$ is associative and has identity 1-cells $\bot_X = \top_X^\perp\colon X\rarr X$:
\begin{align*}
A\op (B\op C) &= A\op (C^\perp\otimes  B^\perp)^\perp \cong  ((C^\perp\otimes  B^\perp)\otimes A^\perp)^\perp  \\
&\cong(C^\perp\otimes  (B^\perp\otimes A^\perp))^\perp \cong    (C^\perp\otimes  (A\op B)^\perp)^\perp \\
&= (A\op B)\op C
\end{align*}
and  $\top_X^\perp\op A  \cong A\cong A\op \top_Y^\perp$.\\

To see that \cB\ is a linear bicategory, we will define the left linear distribution 
\[A\ot(B\op C)\rarr (A\ot B)\op C\] 
Since $A\ot (B\op C)\cong A\ot(B^\perp\rTo\,C)$ and $(A\otimes B)\op C \cong (A\ot B)^\perp\rTo\,C\cong (B^\perp\,\lTo A)\rTo\,C$, by Lemma~\ref{lemma}; it suffices to define a 2-cell 
\[A\ot(B^\perp\rTo\,C)\rarr (B^\perp\,\lTo A)\rTo\,C\] 
or equivalently, its transpose 
\[ (B^\perp\,\lTo A)\ot (A\ot (B^\perp\rTo\, C))\rarr C\] 
For this, we can use the associator and the evaluation maps 
\[(B^\perp\,\lTo A)\ot (A\ot (B^\perp\rTo\,C))\rarr((B^\perp\,\lTo A)\ot A)\ot (B^\perp \rTo\,C)\rarr B^\perp\ot (B^\perp \rTo\, C)\rarr C\]
Similarly, we get a 2-cell $(A\op B)\ot C\rarr A\op (B\ot C)$, as desired. It remains to show the coherence conditions, in particular the ones involving the linear distributions.
\end{proof}

The last section of this article will be a discussion of several examples of non-posetal cyclic $*$-autonomous bicategories, including $\biquant$\ and $\qtld$. \\

Cyclic $*$-autonomous bicategories are unique in that every 1-cell has a canonical cyclic linear adjoint, in the following sense:
\begin{defn}\cite{Cockett_Koslowski_Seely_2000}
A {\em linear adjunction} $A\Ladj B\colon X\rarr Y$ consists of a pair of 1-cells $A\colon X\rarr Y$ and $B\colon Y\rarr X$, equipped with 2-cells $\tau\colon \top_X\rarr A\op B$, known as the unit, an $\gamma\colon B\ot A\rarr \bot_Y$, known as the conunit, satisfying the snake equations. 
\end{defn}

\begin{defn}\cite{Cockett_Koslowski_Seely_2000}
A {\em cyclic linear adjunction} $A\CLadj B$ is a pair of linear adjoints $A\Ladj B$ and $B\Ladj A$.
\end{defn}

\begin{prop}\cite{Cockett_Koslowski_Seely_2000}
Any two right (respectively left) linear adjoints to a 1-cell are isomorphic, in the sense that there is a unique 2-cell mediating the isomorphism. 
\end{prop}

Then, 
\begin{prop}\cite{Cockett_Koslowski_Seely_2000}
Consider a cyclic $*$-autonomous bicategory, then every 1-cell $A\colon X\rarr Y$ has a unique (up to isomorphism) cyclic linear adjoint given by $A^\perp\colon Y\rarr X$.
\end{prop}

\section{\LD-quantales and linear quantaloids} \label{LDquantales_linearquantaloids}

It is immediate that Girard quantaloids are locally posetal cyclic $*$-autonomous bicategories. In particular, the second composition is defined as the dual of the first. An alternate approach to the model theory of linear logic is that of linear bicategories where the tensor and par are taken as primitive. We define what the analogous structure would be for quantaloids below.

\begin{defn}\label{Linear_Quantaloid}
A {\em linear quantaloid} is a locally small category \cQ\ whose hom-sets are complete lattices with binary operations $\ot$ and $\op$, and families of distinguished morphisms $\{\top_a~|~a\in {\rm ob}\cQ\}$ and $\{\bot_a~|~a\in {\rm ob}\cQ\}$ such that
\begin{itemize}
    \item $(\cQ,\ot,\top_a)$ and $(\cQ^{co},\op,\bot_a)$ are quantaloids,
    \item for all $f\colon a\rarr b, g\colon b\rarr c, h\colon c\rarr d\in\cQ$, \[ f\ot(g\op h)\leq(f\ot g)\op h\quad {\rm and}\quad (f\op g)\ot h\leq f\op (g\ot h)\]
\end{itemize}
\end{defn}

\noindent Every Girard quantaloid is a linear quantaloid and we have the following obvious observation. 

\begin{lem}
A linear quantaloid is a linear bicategory.
\end{lem}

\begin{defn}
If \cQ\ and \cQ' are linear quantaloids, a {\em linear quantaloid homomorphism} $F=(F_\ot, F_\op)\colon\cQ\rarr\cQ'$ is a linear functor such that $F_\ot$ and $F_\op$ are quantaloid homomorphisms. 
\end{defn}

A main result of the present work will be new examples of linear bicategories which are linear quantaloids. To construct these, we need the definition of the analogous linear structure for quantales. These are the \LD-quantales as defined below. 

\begin{defn}
An {\em \LD-quantale} $(Q, \ot, \top, \op, \bot)$ is a complete lattice $Q$ with operations $\ot$ and $\op$ and elements $\top$ and $\bot$ such that 

\begin{itemize}
\item $(Q,\ot, \top)$ and $(Q^{op},\op,\bot)$ are quantales.
\item for all $a, b, c\in Q$, \[ a\ot(b\op c)\leq(a\ot b)\op c\quad {\rm and}\quad (a\op b)\ot c\leq a\op (b\ot c)\]
\end{itemize}
\end{defn}

Clearly a Girard quantale is a LD-quantale.\\

The notion of an \LD-quantale is not truly a new one. It has previously appeared in some form in the literature, following the introduction of LDCs by \cite{Cockett_Seely_1997}. In particular, it has been considered within the field of algebraic logic, when discussing ordered algebras, wherein linear distributivity was called {\em hemi-distributivity} by \cite{Dunn_Hardegree_2001}. \\

Indeed, within this context, a \LD-quantale refers to a lattice-ordered bimonoid $\langle \mathbb{A}, \wedge, \vee, \cdot, 1, +, 0\rangle$, where all joins and meets are admissible in the multiplicative pomonoid $\mathbb{A}_\cdot$ and in the additive pomonoid $\mathbb{A}_+$ respectively, as defined by \cite{Galatos_Prenosil_2023}. \\

Now, every locale is a quantale and \cite{Rosenthal_1992} remarks that a locale is a Girard quantale if and only if it is a Boolean algebra with $\qzero$ its dualizing element. We can extend this remark to \LD-quantales, but fist we must introduce a slightly less well-known type of lattice.

\begin{defn}\cite{Reyes_Zolfagheri_1996}
\begin{itemize}
    \item A {\em Heyting} algebra is a bounded distributive lattice $\cL$ with an ``implication'' operator $\rarr\,\colon \cL\times\cL\rarr\cL$ with the following property $\forall a,b,c\in \cL$:
    \[ a\leq b\rarr c \iff a\wedge b\leq c\]
    
    \item A {\em co-Heyting} algebra is a bounded distributive lattice $\cL$ with an ``subtraction'' operator $\backslash\,\colon \cL\times\cL\rarr\cL$ with the following property $\forall a,b,c\in \cL$:
    \[ a\backslash b\leq c \iff a\leq b\vee c\]

    \item A {\em bi-Heyting} algebra is a bounded distributive lattice that is both a Heyting and a co-Heyting algebra.
\end{itemize}
\end{defn}

Then:

\begin{prop}
A locale is a \LD-quantale $(\cL, \wedge, \qone, \vee, \qzero)$ if and only if it is a complete bi-Heyting algebra.
\end{prop}
\begin{proof}
Note that a complete Heyting algebra is the same notion as a locale. Moreover, a locale is a \LD-quantale if and only if the opposite infinitary law holds,
\[ (\bigwedge_{i\in I} a_i)\vee b = \bigwedge_{i\in I} (a_i\vee b)\]
This is equivalent to requiring a right adjoint to $(-)\vee a\colon \cL^{op}\times\cL^{op}\rarr\cL^{op}$, in other words $\cL$ has a subtraction operation and is a co-Heyting algebra. 
\end{proof}

\begin{ex}
\begin{enumerate}
	\item $\sREL(X)$, the poset of relations on a set $X$, is a Girard quantale with cyclic dualizing element $\Delta_X^c$, see Proposition 1.2 in \cite{Rosenthal_1992}, and therefore is a LD-quantale.

	\item Consider the following locales which are \LD-quantales.

	\begin{itemize}
		\item Two-chain $\Omega$ is a Girard quantale, and therefore a LD-quantale.
		\item Three-chain $3$ is the smallest locale, which is not Boolean as the law of excluded middle doesn't hold: \[ 1/2 \vee (0\,\lTo 1/2) = 1/2 \vee 0 = 1/2 \neq 1 \]
		As such it is not a Girard quantale, but it is nonetheless a LD-quantale as it is a bi-Heyting algebra.
		\item ${\sf P_{max}}$ is a LD-quantale, as it is a bi-Heyting algebra, with $\op = {\rm min}$ and $\bot = \infty$, but it is not a Girard quantale as $\infty$ is not a dualizing element:
\[ {\rm if}\quad a\neq\infty, \quad a\rTo\,\infty= \infty \implies (a\rTo\,\infty)\rTo\,\infty = 0 \]
            \item \cite{Reyes_Zolfagheri_1996} demonstrate that given an oriented irreflexive multigraph, the lattice of its subgraphs is a bi-Heyting algebra, and therefore an \LD-quantale.
            \item \cite{Borceaux_Bourn_Johnstone_2006} defined the notion of a bi-Heyting topos to be a topos such that the lattice of its subobjects is precisely a bi-Heyting algebra. An important example is the topos of presheafs $[\cC^{op}, Set]$ for any small category $\cC$. As such, its lattice of subobjects is a \LD-quantale.
	\end{itemize}
	
	\item Consider the unit interval with multiplication $([0,1],\cdot, 1)$. It becomes a LD-quantale $([0,1],\cdot,\op)$ when considering truncated addition $a\op b ={\sf min}(a+b, 1)$ for its par structure with unit $\bot = 0$. $([0,1]^{op}, \op, 0)$ is a quantale since \[ \bigwedge_{\alpha} {\sf min}(a+b_\alpha, 1) = {\sf min}(a+\bigwedge_\alpha b_\alpha, 1) \] Linear distributivities hold as \[ a\cdot(b+c) = (a\cdot b) + (a\cdot c) \leq (a\cdot b)+c \quad\quad \forall a\in[0,1] \]
	It is not however a Girard quantale as $0$ is not a dualizing element: \[ {\rm if}\quad a\neq 0, \quad a\rTo\,0= 0 \implies (a\rTo\,0)\rTo\,0 = 1 \]

	Since $([0,1],\cdot, 1)\,\cong\, {\sf P_{+}}$, we can use the isomorphims to define a par structure on ${\sf P_{+}}$: \mbox{$a\op b = {\sf max}(-{\sf ln}({\sf exp}(-a)+{\sf exp}(-b)), 0)$} and $\bot = \infty$. Then, ${\sf P_{+}} = ([0,\infty]^{op}, +, \op)$ is a LD-quantale isomorphic to $([0,1],\cdot,\op)$.
	
\end{enumerate}
\end{ex}

\begin{rem}
Given a locale $L$, its booleanization ${\sf Bool}(L)=\{a\in L\,|\, (a\rTo\,\qzero)\rTo\,\qzero = a\}$ is a complete Boolean algebra by \cite{Banaschewski_Pultr_1996}. Therefore, viewing any locale as a LD-quantale, there is always a sub-locale ${\sf Bool}(L)$ which is a Girard quantale. 
\end{rem}

The following is another example of an LD-quantale which requires a little more discussion.

\subsection{Shift monoids}

\cite{Cockett_Seely_1997} introduce the notion of a {\it shift monoid}. These provide examples 
of discrete linearly distributive categories which are not $*$-autonomous. 

\begin{defn}
\begin{itemize}

\item A {\em shift monoid} consists of a 4-tuple $\cM=(M,+,\top,a)$ where $(M,+,\top)$ is a commutative monoid and $a$ is an 
invertible element of $M$. 
\item If \cM\ is a shift monoid, define a second multiplication by
\[x\cdot y=x+y-a\]
Then this is a second monoid structure on $M$ with unit given by $a$.
\end{itemize}
\end{defn}

It is trivial to see that $x\cdot(y+z)=(x\cdot y)+z$ and so every shift monoid is a discrete linearly distributive category $(M,+,\cdot)$. We modify the notion of shift monoid as follows in order to construct \LD-quantales. Recall that a commutative monoid $M$ is {\it cancellative} if for all $a,b,c\in\ M$, one has $$a+b=a+c\Rightarrow b=c$$

\begin{prop}
\begin{itemize}
\item Let $(M,+,\top)$ be a cancellative commutative monoid. 
We view $M$ as a discrete poset and then add top and bottom
elements which we denote \qone\ and \qzero. Denote this set as $M^+$, We then extend the addition on $M$:

\[\qone+b=\qone\,\mbox{if $b\in M$ or $b=\qone$} \,\,\,\,\,\,\,\,\,\,\,\, \qzero+b=\qzero\mbox{for all $b\in M^+$}\]
\noindent Then $(M^+,+,\top)$ is a commutative quantale.

\item Let $\cM = (M,+,\top, a)$ be furthermore a cancellative commutative shift monoid, then extend the second operation in the dual way:
\[\qzero\cdot b=\qzero\mbox{    if $b\in M$ or $b=\qzero$} \,\,\,\,\,\,\,\,\,\,\,\, \qone\cdot b=\qone\mbox{     for all $b\in M^+$}\]

Then $(M^+, +,\cdot)$ is a \LD-quantale.
\end{itemize}
\end{prop}

\begin{rem} The cancellative property is needed to ensure that $+$ preserves all suprema in $M^+$.
\end{rem}

\section{\QRel\ as a linear bicategory}\label{QRel}

We now demonstrate that if $Q$ is a Girard quantale, or more generally an \LD-quantale, \QRel\ determines a linear quantaloid, providing new examples of linear bicategories.

\begin{prop}  
\QRel\ is a Girard quantaloid with cyclic dualizing family $\cD=\{\bot_X\colon X\tov X\}$ if and only if $Q$ is a Girard quantale with cyclic dualizing element $\bot$, where $\bot = \bot_1(*,*)$, $\bot_1\colon 1\tov1$ being the constant map between singleton sets $1=\{*\}$, and

\[\bot_X(x,x')=\begin{cases}\bot  & x=x' \\ \qone & x\ne x'
\end{cases}\]
\vskip -.5in \hfill $(*)$ \\
\end{prop}

\vskip .2in 

\begin{proof} Suppose \QRel\ is a Girard quantaloid. Given  $a\in Q$, let $R_a\colon 1\tov1$ denote the $a$-valued constant relation. Since $R_a \rTo \bot_1=\bot_1 \lTo  R_a $ and $R_a^{\perp\perp}=R_a$, it follows that $\bot=\bot_1(*,*)$ is a cyclic dualizing element for $Q$, and so $Q$ is a Girard quantale. \\

\noindent Consider a set $X$. Let $x'\in X$ and define $Q$-relation $R\colon 1\tov X$ by 

\[R(*,x)=\begin{cases} \top & x=x' \cr \qone & x\ne x'\end{cases}\]

\noindent Now $(\bot_X  \lTo  R)(x,*)=\displaystyle\bigwedge_{\bar x} \bot_X(x,\bar x)\lTo  R(*,{\bar x})= \bot_X(x,x')$ and 
\[(R \rTo  \bot_1)(x,*)=R(*,x) \rTo \bot= \begin{cases}  \bot& x=x' \\ \qone & x\ne x'\end{cases}\] 
 
\noindent since $\top\rTo\bot=\bot$. As \cD is cyclic, $\bot_X\lTo  r = r \rTo \bot_1$ and consequently, (*) holds. \\

Suppose $Q$ is a Girard quantale with cyclic dualizing element $\bot$. Define a family of $Q$-relations \cD by $(*)$. Consider $R\colon X\tov Y$, then $$(R \rTo\bot_X)(y,x) =  \displaystyle \bigwedge_{\bar x} R(\bar x,y)\rTo \bot_X(\bar x,x)=R(x,y)\rTo\bot = R(x,y)^\perp$$  $$(\bot_Y  \lTo  r)(y,x) = \displaystyle \bigwedge_{\bar y} d_Y(y,\bar y)  \circ\!\!-\!\!\!\!-  r(x,
\bar y)= \bot \lTo  R(x,y) = R(x,y)^\perp$$  and so \cD is a dualizing family for \QRel\ as $R\rTo\bot_X = \bot_Y\lTo R$. \cD being cyclic follows immediately and as such \QRel\ is a Girard quantaloid.
\end{proof}

Consequently, if $Q$ is a Girard quantale, there is a second categorical structure on \QRel\ determining a linear bicategory. Indeed its second composition is obtained as the de Morgan dual of its standard composition
\[ R\op S(x,z) = (\bigvee_{y\in Y} S(y,z)^\perp \ot R(x,y)^\perp)^\perp = \bigwedge_{y\in Y} (S(y,z)^\perp \ot R(x,y)^\perp)^\perp = \bigwedge_{y\in Y}R(x,y)\op S(y,z)\]
with identities $\bot_X$. \\

This generalizes further. Suppose $(Q,\ot,\top)$ and $(Q^{op},\op,\bot)$ are quantales. Two notions of composition \[\ot,\op\colon \cB(X,Y)\times\cB(Y,Z)\lrarr\cB(X,Z)\]

\noindent are defined as follows: given $R\colon X\tov Y$ and $S\colon Y\tov Z$,
\[R\ot S(x,z)=\bigvee_{y\in Y}(R(x,y)\ot S(y,z)) \mbox{\,\,\,\,\,\,\,\,\,\,\,\,\,\,\,\,\,\,\,\,}R\op S(x,z)=\bigwedge_{y\in 
Y}(R(x,y)\op S(y,z))\]

\noindent The identity 1-cells are given by:

\[\top_X(x,x')=\begin{cases}\top& \mbox{if } x=x'  \\ \qzero &\mbox{if } x\neq x' \end{cases}\mbox{\,\,\,\,\,\,\,\,\,\,\,\,\,\,\,\,\,\,\,\,\,\,\,}
\bot_X(x,x')=\begin{cases}\bot& \mbox{if } x=x' \\ \qone &\mbox{if } x\neq x'  \end{cases}\] and we get:
\begin{thm}\label{LDQ} $Q$ is an \LD-quantale if and only if \QRel\ is a linear quantaloid. 
\end{thm}

\begin{proof} By Lemma \ref{QRel_quantaloid}, $(\QRel, \ot,\top_X)$ and $(Q^{op}$-$\sREL, \op, \bot_X)\,\cong\, (\QRel^{co}, \op,\bot_X)$ are quantaloids as $(Q,\ot,\top)$ and $(Q^{op},\op,\bot)$ are quantales, inheriting their structure from $Q$ point-wise. \\

Suppose $($Q$,\ot, \op)$  is an \LD-quantale. Given \[ W\stackrel{R}{\tov}X\stackrel{S}{\tov}Y\stackrel{T}{\tov}Z\] we have $R\otimes(S\op T) \le (R\otimes S)\op T$ if and only if, for all $w,z$, \[ \bigvee_x 
[R(w,x)\ot(S\op T)(x,z)] \le \bigwedge_y [(R\otimes S)(w,y)\op T(y,z)] \]
if and only if  $R(w,x)\ot(S\op T)(x,z) \le (R\otimes S)(w,y)\op T(y,z)$, for all $w,x,y,z $. But,
\begin{eqnarray*}
R(w,x)\ot(S\op T)(x,z) 
& = &R(w,x)\ot \bigwedge_ y [S(x,y) \op T(y,z)] \\
& \le & R(w,x)\ot[S(x,y) \op T(y,z)] \\
& \le & [R(w,x)\ot S(x,y)] \op T(y,z) \\
& \le &\bigvee_x[R(w,x)\ot S(x,y)] \op T(y,z) \\
& \le &(R\otimes S)(w,y)\op T(y,z) 
\end{eqnarray*}
\noindent The other inequality follows similarly and we conclude that \QRel\ is a linear quantaloid. 

Conversely, suppose \QRel\ is a linear quantaloid. Then $a,b,c$ in $Q$ induces 
$1\stackrel{R_a}{\tov}1\stackrel{R_b}{\tov}1\stackrel{R_c}{\tov}1$ in \QRel, and $R_a\otimes(R_b\op R_c) \le (R_a\otimes R_b)\op R_c$ implies $a\ot(b\op c)\le (a\ot b)\op c$. Thus $Q$ is an \LD-quantale.
\end{proof}

\QRel\ is a linear quantaloid for all locales $Q$ as they are LD-quantales. 
\begin{ex} 
\begin{enumerate}

\item \mbox{$2$-\sREL}$\,\cong\,\sREL$ is a Girard quantaloid, with its par composition given by de Morgan duality, the same as previously introduced in the preliminaries section.
	
\item \mbox{$3$-\sREL} is a linear quantaloid, which is not a Girard quantaloid, as $3$ is a complete bi-Heyting algebra, but not a Boolean algebra. 
	
\item \mbox{${\sf P_{max}}$-\sREL}, the quantaloid of sets and ``extended'' distance relations, is a linear quantaloid, which is not Girard, with its par composition defined, for $D_1\colon X\tov Y$ and $D_2\colon Y\tov X$ by \[ (D_1 \op  D_2)(x,z) = \bigvee_{y\in Y} {\sf min}(D_1(x,y), D_2(y,z)) \]
and par identities given by \[ \bot_X(x,x') = \begin{cases}\infty& \mbox{if } x=x' \\ 0 &\mbox{if } x\neq x'  \end{cases}\]

\item \mbox{$[0,1]$-\sREL} is another linear quantaloid, which is not Girard, of sets and relations, with its par composition defined, for $D_1\colon X\tov Y$ and $D_2\colon Y\tov X$ by \[ (D_1 \op D_2)(x,z) = \bigwedge_{y\in Y}{\sf min}(D_1(x,y)+D_2(y,z), 1) \]
and par identities given by \[ \bot_X(x,x') = \begin{cases}0& \mbox{if } x=x' \\ 1 &\mbox{if } x\neq x'  \end{cases}\]
	
Alternatively, \mbox{${\sf P_{+}}$-\sREL} with its par composition defined, for $D_1\colon X\tov Y$ and $D_2\colon Y\tov X$ by \[ (D_1 \op D_2)(x,z) = \bigvee_{y\in Y}{\sf max}(-{\sf ln}({\sf exp}(-D_1(x,y))+{\sf exp}(-D_2(y,z))), 0) \]
and par identities given by \[ \bot_X(x,x') = \begin{cases}\infty& \mbox{if } x=x' \\ 0 &\mbox{if } x\neq x'  \end{cases}\]

\end{enumerate}
\end{ex}

\section{Enriched in a linear quantaloid}\label{Enriched_linearquantaloid}

\cite{Rosenthal_1992} demonstrated that the definition of a Girard quantaloid was closed under various constructions, in particular if \cQ\ is a Girard quantaloid, then \QMod\ is one as well. This can be easily extended into a necessary and sufficient condition.

\begin{prop}
\QMod\ is a Girard quantaloid with cyclic dualizing family $\{\bott_M\colon M\tov M\}$ if and only if \cQ\ is a Girard quantaloid with cyclic dualizing family  $\{\bot_a\colon a\rarr a\}$, where \mbox{$\bot_a=\bott_{M_a}(*, *)$}, $M_a=(1,\rho_a)$ being the \cQ-category defined by $\rho_a(*)=a$ and $M_a(*,*)=\top_a$, and
$$\bott_M(x,x')=M(x',x)\rTo \bot_{\rho(x')}$$
\end{prop}

\begin{proof}
Suppose \QMod\ is a Girard quantaloid with cyclic dualizing family $\{\bott_M\colon M\tov M\}$. Given an object $a$ in \cQ, define the  \cQ-category $M_a=(1,\rho_a)$ as indicated above. Define a family of morphisms $\cD=\{\bot_a\colon a\rarr a\}$ in \cQ\ by $\bot_a=\bott_{M_a}(*, *)$. \\

\noindent Given a morphism $f\colon a\rarr b$ in \cQ, consider its image $\Theta_f\colon M_{a}\tov M_{b}$ under the embedding of \cQ\ into \QMod. Then,
\[ f\rTo \bot_a = \Theta_f(*,*)\rTo~ \bott_{M_a}(*,*) = \bigwedge_{x\in \{*\}} \Theta_f(x,*)\rTo~ \bott_{M_a}(x,*) = (\Theta_f \rTo~ \bott_{M_a})(*, *) = \Theta_f^\perp(*,*) \]
\[ \bot_b \lTo f = \bott_{M_b}(*,*) ~\lTo \Theta_f (*,*) = \bigwedge_{x\in \{*\}}\bott_{M_b}(*,x)~\lTo \Theta_f(*,x) = (\bott_{M_b}~\lTo \Theta_f)(*, *) = \Theta_f^\perp(*,*) \]
implying \cD is a cyclic family of morphisms in \cQ\ and \cD being dualizing follows similarly. \\

The sufficiency proof is shown in Theorem 3.1 in \cite{Rosenthal_1992}. 
\end{proof}

Notice that if \cQ\ is a Girard quantaloid, then for each \cQ-category $M=(X,\rho)$ and $(x,x')\in X$, there is a morphism \mbox{$M(x',x)^\perp\colon\rho(x)\rightarrow\rho(x')$} satisfying $\forall x, x', x''\in X$
\[ M(x,x)^\bot\leq\bot_{\rho(x)} \qquad M(x,x'')^\bot\leq M(x',x'')^\bot\op M(x,x')^\bot\]
Therefore the map $M^\perp\colon (x,x')\mapsto M(x',x)^\perp$, coupled with $(X,\rho)$ is a $\cQ^{co}$-category, where $\cQ^{co}$ is the quantaloid with the de Morgan dual composition $\op$ and identities $\bot_X$. \\

\noindent $M^\perp$ is moreover a \cQ-module $M\tov M$: $\forall x, x', x''\in X$
\[ M(x',x)^\bot\ot M(x',x'')\leq M(x'',x)^\bot \qquad M(x,x')\ot M(x'',x')^\bot\leq M(x'',x)^\bot \]

\noindent While $M$ becomes a $\cQ^{co}$-module $M^\perp\tov M^\perp$: $\forall x, x', x''\in X$
\[ M(x,x'')\leq M(x',x)^\bot\op M(x',x'') \qquad M(x,x'')\leq M(x,x')\op M(x'',x')^\bot \]

\noindent All together, this entails that $(M, M^\bot)\colon X\rarr\cQ$ is a linear functor, when $X$ is viewed as a degenerate indiscrete linear bicategory. \\
    
\noindent Each \cQ-module $\Theta\colon M\tov N$ interacts coherently with this structure and it determines a $\cQ^{co}$-module $M^\perp\tov N^\perp$: \mbox{$\forall x, x'\in X, y, y'\in Y$}
\[\Theta(x,y)\leq M(x',x)^\bot\op\Theta(x',y) \qquad \Theta(x,y)\leq \Theta(x,y')\op N(y,y')^\bot\]
    
\noindent As with every Girard quantaloid, we can define a second composition on \QMod. Given \mbox{$M\stackrel{\Theta}{\tov}N \stackrel{\Pi}{\tov}P$}, $\Theta\ot\Pi\colon M\tov P$ and $\Theta\op\Pi\colon M\tov P$ are defined by $$(\Theta\ot\Pi)(x,z)=\bigvee_{y\in Y}\Theta(x,y)\ot\Pi(y,z)\quad{\rm and}\quad(\Theta\op\Pi)(x,z)=\bigwedge_{y\in Y}\Theta(x,y)\op\Pi(y,z)$$
    
\noindent The identities 1-cells for  $\ot$ and $\op$ are given by $\ttop_M$ and $\bott_M^\bot$ respectively, where 
\[\ttop_M(x,x')= M(x,x')\quad{\rm and}\quad \bott_M(x,x')=M(x',x)^\bot\]

The above discussion motivates the following new definitions which generalize \QMod\ to the case where \cQ\ is a linear quantaloid as follows.\\

\noindent Let $(\cQ, \ot, \top_a)$ and $(\cQ^{co},\op,\bot_a)$ be quantaloids. 
\begin{defn}
A {\em linear \cQ-category} is a pair $M = (X, \rho)$ where:
\begin{itemize}
	\item X is a set.
	
	\item $\rho$ is a function $X\rightarrow ob\cQ$,
	
	\item there is a function, called the $\ot${\em-enrichment}, assigning a morphism $M_\ot(x,x')\colon\rho(x)\rightarrow\rho(x')$ in \cQ\ to each pair $(x,x')\in X\times X$ such that $\forall x, x', x''\in X$
	\[ \top_{\rho(x)}\leq M_\ot(x,x) \qquad M_\ot(x,x')\ot M_\ot(x',x'')\leq M_\ot(x,x'')\]
	
	\item there is a function, called the $\op${\em-enrichment}, assigning a morphism $M_\op(x,x')\colon\rho(x) \rightarrow\rho(x')$ in \cQ\ to each pair $(x,x')\in X\times X$ such that $\forall x, x', x''\in X$
	\[ M_\op(x,x)\leq\bot_{\rho(x)} \qquad M_\op(x,x'')\leq M_\op(x,x')\op M_\op(x',x'') \]
	
	satisfying the following inequalities $\forall x, x', x''\in X$,
	\[ M_\ot(x,x'')\leq M_\op(x,x')\op M_\ot(x',x'') \qquad M_\ot(x,x'')\leq M_\ot(x,x')\op M_\op(x',x'')\]
	\[ M_\ot(x,x')\ot M_\op(x',x'')\leq M_\op(x,x'') \qquad M_\op(x,x')\ot M_\ot(x',x'')\leq M_\op(x,x'') \]    
\end{itemize}

More succinctly, if \cQ\ is a linear quantaloid, a linear \cQ-category is a linear functor from $X$ to \cQ, where $X$ is viewed as a degenerate indiscrete linear bicategory.
\end{defn}

Note that given a linear \cQ-category $M=(X,\rho)$, $M_\ot = (X,\rho)$ is a $\cQ$-category and $M_\op = (X,\rho)$ is a $\cQ^{co}$-category. Moreover, the $\ot$-enrichment and $\op$-enrichment together assign cyclic linear adjoints in \cQ\ as follows.
\begin{prop}
Consider a linear \cQ-category $M=(X,\rho)$ and a pair $(x,x')\in X\times X$, then  $\ot$-enrichment and $\op$-enrichments provide a cyclic linear adjunctions $M_\ot(x,x')\CLadj M_\op(x',x)\colon \rho(x)\rarr \rho(x')$.
\end{prop}
\begin{proof}
To show $M_\ot(x,x')\Ladj M_\op(x',x)$, we need only provide the unit and co-unit 2-cells, which are inequalities in this context.
\[ \top_{\rho(x)}\leq M_\ot(x,x) \leq M_\ot(x,x')\op M_\op(x',x)\]
\[ M_\op(x',x)\ot M_\ot(x,x')\leq M_\op(x',x') \leq \bot_{\rho(x')}\]
Similarly, $M_\op(x',x)\Ladj M_\ot(x,x')$. 
\end{proof}

\begin{defn}
Consider linear \cQ-categories $M=(X, \rho_M)$ and $N=(Y, \rho_N)$. A {\em linear \cQ-module} $\Theta\colon M\tov N$ consists of an assignment of a morphism $\Theta(x,y)\colon \rho_M(x)\rightarrow\rho_N(y)$ to every pair $(x,y)\in X\times Y$ such that $\forall x, x'\in X, y, y'\in Y,$
\[ \Theta(x,y) \ot N_\ot(y,y')\leq \Theta(x,y') \qquad M_\ot(x,x')\ot\Theta(x',y) \leq \Theta(x,y) \]
\[ \Theta(x,y)\leq M_\op(x,x')\op\Theta(x',y) \qquad \Theta(x,y)\leq\Theta(x,y')\op N_\op(y',y)\]
\end{defn}

The above definition of a linear \cQ-module can be simplified, as the four inequalities are not independent: the top two imply the bottom two, and vice versa. 
\begin{prop}\label{prop:linear_Q_module}
Consider linear \cQ-categories $M=(X, \rho_M)$ and $N=(Y, \rho_N)$, then $\Theta\colon M\tov N$ is a {\em linear \cQ-module} if one of the following conditions holds:
\begin{enumerate}
    \item $\Theta$ is a \cQ-module $M_\ot\tov N_\ot$, and
    \item $\Theta$ is a $\cQ^{co}$-module $M_\op\tov N_\op$.
\end{enumerate}
\end{prop}
\begin{proof}
Suppose $\Theta$ is a \cQ-module $M_\ot\tov N_\ot$, in other words $\Theta(x,y) \ot N_\ot(y,y')\leq \Theta(x,y')$ and $M_\ot(x,x')\ot\Theta(x',y) \leq \Theta(x,y)$ holds  $\forall x, x'\in X, y, y'\in Y$. Then,
\begin{align*}
\Theta(x,y) &= \Theta(x,y)\ot \top_{\rho_N(y)}\leq  \Theta(x,y)\ot N_\ot(y,y) \leq \Theta(x,y)\ot (N_\ot(y,y')\op N_\op(y',y)) \\
&\leq (\Theta(x,y)\ot N_\ot(y,y'))\op N_\op(y',y) \leq \Theta(x,y')\op N_\op(y',y) \\ \\
\Theta(x,y)& = \top_{\rho_X(x)} \ot \Theta(x,y) \leq M_\ot(x,x)\ot \Theta(x,y) \leq (M_\op(x,x')\op M_\ot(x',x))\ot \Theta(x,y) \\
&\leq M_\op(x,x')\op (M_\ot(x',x)\ot \Theta(x,y)) \leq M_\op(x,x')\op \Theta(x',y)
\end{align*}
The other direction follows similarly. 
\end{proof}

\begin{rem}
It was shown by Cockett, Koslowski and Seely that representable poly-bicategories are linear bicategories and poly-functors between them are linear functors. When viewing linear \cQ-categories as a poly-functors between representable poly-bicategories, a linear \cQ-module is a {\it poly-module} between the poly-functors, as defined in Def 4.1 \cite{Cockett_Koslowski_Seely_2003}.
\end{rem}

We can now define bicategory structures attached to the above data. 

\begin{defn}
Define \QMod\ to consist of the following data:
\begin{itemize}
	\item 0-cells are linear \cQ-categories
	
	\item given a pair of \cQ-categories $M$ and $N$, there is a category with 1-cells the linear \cQ-modules $M\tov N$ and 2-cells the point-wise inequalities, i.e.,
	\[\Theta\leq\Pi\Leftrightarrow~\forall (x,y)\in X\times Y, \quad\Theta(x,y)\leq\Pi(x,y)\]
	
	\item given \cQ-categories $M$, $N$ and $P$, there are two composition functors $\ot, \op$ defined for $\Theta\colon M\tov N$ and $\Pi\colon N\tov P$ by
	\[(\Theta\ot\Pi)(x,z)=\bigvee_{y\in Y}\Theta(x,y)\ot\Pi(y,z)\quad\quad{\rm and}\quad\quad (\Theta\op\Pi)(x,z)=\bigwedge_{y\in Y}\Theta(x,y)\op\Pi(y,z)\]
	
	\item given every \cQ-category $M$, there are identity 1-cells $\ttop_M, \bott_M\colon M\tov M$ defined by 
	\[\ttop_{M}(x,x')=M_\ot(x,x')\quad\quad{\rm and}\quad\quad\bott_{M}(x,x')=M_\op(x,x')\]
\end{itemize}
\end{defn}

And we can show:

\begin{thm}\label{QMod_linear}
$(\QMod, \ot, \op)$ is a linear quantaloid if and only if $(\cQ,\ot,\op)$ is a linear quantaloid. 
\end{thm}

This will be an immediate consequences of similar results for linear \cQ-matrices and linear monads in \cQ, so we will delay proving the above theorem. \\

Another construction that was shown to be Girard by \cite{Rosenthal_1992} is the quantaloid of \cQ-matrices:

\begin{prop}
\MatrQ\ is a Girard quantaloid with cyclic dualizing family $\{\bott_{(X,\gamma)}\colon (X,\gamma)\tov (X,\gamma)\}$ if and only if \cQ\ is a Girard quantaloid with cyclic dualizing family $\{\bot_a\colon a\rarr a\}$, where \mbox{$\bot_a={\bott_{(1,\gamma_a)}}_{*,*}$}, $(1,\gamma_a)$ consisting of the singleton set $1$ and the function $\gamma_a\colon *\mapsto a$, and
\[{\bott_{(X,\gamma)}}_{x,x'}=\begin{cases}\bot_{\gamma(x)}  & x=x' \\ \qone_{\gamma(x),\gamma(x')} & x\ne x'
\end{cases}\]
\vskip -.5in
\end{prop}
\begin{proof}
The forward direction follows similarly to the proof in the case of \QMod, while the backwards direction proved by Theorem 3.2 in \cite{Rosenthal_1992}. 
\end{proof}

If $(\cQ, \ot, \top_a)$ and $(\cQ^{co},\op,\bot_a)$ are quantaloids, then \MatrQ\ inherits a second bicategorical structure. The two notions of composition are defined as follows: given \cQ-matrices \mbox{$(X, \gamma)\stackrel{r}{\tov}(Y,\phi)\stackrel{s}{\tov}(Z,\chi)$}, $r\ot s, r\op s\colon(X,\gamma)\tov (Z,\chi)$ are defined by 
\[ (r\ot s)_{x,z}=\bigvee_{y\in Y}r_{x,y}\ot s_{y,z}\quad\quad{\rm and}\quad\quad (r\op s)_{x,z}=\bigwedge_{y\in Y}r_{x,y}\op s_{y,z} \]
Identity 1-cells $\ttop_{(X,\gamma)}, \bott_{(X,\gamma)}\colon (X,\gamma)\tov (X,\gamma)$ are defined by 
\[{\ttop_{(X,\gamma)}}_{x,x'}=\begin{cases}\top_{\gamma(x)}& \mbox{if } x=x' \\  \qzero_{\gamma(x),\phi(x')} &\mbox{if } x\neq x' \end{cases}
\quad\quad{\rm and}\quad\quad {\bott_{(X,\gamma)}}_{x,x'}=\begin{cases} \bot_{\gamma(x)}& \mbox{if } x=x' \\  \qone_{\gamma(x),\phi(x')} &\mbox{if } x\neq x' \end{cases}\]

\begin{thm}
\MatrQ\ is a linear quantaloid if and only if \cQ\ is a linear quantaloid.
\end{thm}
\begin{proof}
The proof is identical to that of Theorem \ref{LDQ}, replacing objects in a quantale $Q$ by morphisms in \cQ.
\end{proof}

\begin{rem} As in the case of ordinary quantales, it is immediate that for an LD-quantale $Q$, the linear quantaloid ${\sf Matr}\mathcal{ B}(Q)$ is isomorphic to $\QRel$. 
\end{rem}

Finally, as one might expect, taking monads and modules in a Girard quantaloid remains Girard.

\begin{prop}
\MonQ\ is a Girard quantaloid with cyclic dualizing family $\{\bott_{(a,m)}\colon (a,m)\tov (a,m)\}$ if and only if \cQ\ is a Girard quantaloid with cyclic dualizing family $\{\bot_a\colon a\rarr a\}$, where $\bot_a = \bott_{(a,\top_a)}$, $(a,\top_a)$ being the trivial monad on $a$, and $$\bott_{(a,m)} = m\rTo \bot_a$$
\end{prop}
\begin{proof}
The forward direction is immediate from the trivial monad embedding of \cQ\ into \MonQ, while the backwards direction is proven by Theorem 3.3 in \cite{Rosenthal_1992}.
\end{proof}

As in the case of \QMod, this can be generalized to the linear setting.\\

Let $(\cQ, \ot, \top_a)$ and $(\cQ^{co},\op,\bot_a)$  be quantaloids.
\begin{defn}\cite{Cockett_Koslowski_Seely_2000}
A {\em linear monad} $(a,m)=(a,m_\ot,m_\op)$ in \cQ\ is a pair of compatible $\ot$-monad $(a, m_\ot)$ and $\op$-comonad $(a, m_\op)$, i.e., it consists of 
\begin{itemize}
	\item an object $a$
	\item a morphism $m_\ot:a\rightarrow a$ such that \[ \top_a\leq m_\ot \quad\quad {\rm and}\quad\quad m_\ot \ot m_\ot \leq m_\ot\]
	\item a morphism $m_\op:a\rightarrow a$ such that \[m_\op\leq\bot_a \quad\quad {\rm and}\quad\quad m_\op\leq m_\op \op m_\op\] 

satisfying the following inequalities:
\[ m_\ot \leq m_\op \op m_\ot \qquad m_\ot \leq m_\ot \op m_\op\]
\[ m_\ot \ot m_\op \leq m_\op \qquad m_\op \ot m_\ot \leq m_\op\]
\end{itemize}    

\end{defn}

\begin{defn}
Let $(a,m)$ and $(b, n)$ be linear monads in \cQ. A {\em linear $(m,n)$-module} (or monad module) $f\colon(a,m)\tov(b,n)$ is a morphism $f\colon a\rarr b$ in \cQ\ which is $\ot$-monad module $f\colon(a,m_\ot)\tov (b,n_\ot)$ and a $\op$-comonad module $f\colon (a,m_\op)\tov(b,n_\op)$, i.e., satisfies the following inequalities:
\[ f \ot n_\ot\leq f \qquad m_\ot\ot f\leq f \]
\[f \leq m_\op\op f \qquad f \leq f \op n_\op\]
\end{defn}

For the same reasons as in the case of linear \cQ-modules, the $\ot$ and $\op$ monad maps are cyclic linear adjoints and the above definition can be simplified. A morphism $f\colon a\rarr b$ in \cQ\ is $\ot$-monad module $f\colon(a,m_\ot)\tov (b,n_\ot)$ if and only if it is a $\op$-comonad module $f\colon (a,m_\op)\tov(b,n_\op)$.

\begin{defn}
Define \MonQ\ to consist of the following data:
\begin{itemize}
	\item 0-cells are linear monads in \cQ
	
	\item given a pair of linear monads $(a,m)$ and $(b,n)$, there is a category with 1-cells the linear $(m,n)$-modules and 2-cells are inherited from \cQ
	
	\item given linear monads $(a,m)$, $(b,n)$ and $(c,p)$, there are two composition functors $\ot, \op$ which are inherited from \cQ
	
	\item given every linear monad $(a,m)$, there are identity 1-cells $\ttop_{(a,m)}, \bott_{(a,m)}\colon (a,m)\tov (a,m)$ defined by 
	\[\ttop_{(a,m)}=m_\ot \quad\quad{\rm and}\quad\quad \bott_{(a,m)}=m_\op \]
\end{itemize}
\end{defn}

\begin{lem}
If \cQ\ is a linear quantaloid, then $(\MonQ, \ot,\ttop_{(a,m)})$ and $(\MonQ^{co},\op,\bott_{(a,m)})$ are quantaloids.
\end{lem}
\begin{proof}
 $(\MonQ, \ot,\ttop_{(a,m)})$ is a category with a well-defined composition $\ot$: given linear monad modules $f\colon(a,m)\tov(b,n)$ and $g\colon(b,n)\tov(c,p)$, it is immediate that $f\ot g\colon a\rarr c$ is a $\ot$-monad module since $f, g$ are $\ot$-monad modules and, by linear distributivity, $f\ot g\colon a\rarr c$ is a $\op$-comonad module as follows.
 \[ f\ot g\leq (m_\op \op f)\ot g \leq m_\op \op (f\ot g) \quad\quad{\rm and}\quad\quad f\ot g\leq f\ot(g\op p_\op)\leq (f\ot g)\op p_\op \]
 as $f, g$ are $\op$-comonad modules.
Identities $\ttop_{(a,m)}\colon (a,m)\tov (a,m)$ are also well-defined as $m_\ot$ is a linear monad module:
\[ m_\ot \ot m_\ot \leq m_\ot \quad,\quad\quad m_\ot\leq m_\op\op m_\ot \quad\quad{\rm and}\quad\quad m_\ot\leq m_\ot \op m_\op \]
by the definition of linear \cQ-categories. The quantaloid structure is directly inherited from $(\cQ,\ot,\top)$. Similarly, $(\MonQ^{co},\op,\bott_{(a,m)})$ is a quantaloid.
\end{proof}

Then we get this result:
\begin{thm}
\MonQ\ is a linear quantaloid if and only if \cQ\ is a linear quantaloid.
\end{thm}
\begin{proof}
Suppose \cQ\ is a linear quantaloid, then it is immediate that \MonQ\ is a linear quantaloid as compositions $\ot$ and $\op$ are inherited directly. \\

Suppose \MonQ\ is a linear quantaloid, consider the mapping of each object $a\in\cQ$ to the trivial linear monad $(a,\top_a,\bot_a)$ and each morphism $f\colon a\rarr b$ to itself viewed as the trivial linear monad module $f\colon (a,\top_a,\bot_a)\tov (b,\top_b,\bot_b)$. Then, $\forall f\colon a\rarr b, g\colon b\rarr c, h\colon c\rarr d$, we know $f\ot(g\op h)\leq (f\ot g)\op h$ and $(f\op g)\ot h\leq f\op (g\ot h)$, since the inequalities hold when they are viewed as 1-cells in \MonQ. 
\end{proof}

We now return to the proof of Theorem \ref{QMod_linear}:
\begin{proof}
Suppose \cQ\ is a linear quantaloid, then \MatrQ\ is a linear quantaloid and therefore we can perform the linear monad construction and see that ${\sf Mon}\MatrQ$ is a linear quantaloid. By looking at the definitions, it is immediate that $({\sf Mon}\MatrQ, \ot, \ttop_{(X,\gamma), m)})$ is biequivalent to $(\QMod, \ot, \ttop_M)$ and $({\sf Mon}\MatrQ^{co}, \op, \bott_{(X,\gamma), m)})$ is biequivalent to $(\QMod^{co}, \op, \bott_M)$. Thus \QMod\ and $\QMod^{co}$ are quantaloids. \\

Moreover, given linear \cQ-modules $\Theta\colon M\tov N$, $\Pi\colon N\tov P$ and $\Sigma\colon P\tov R$, 
\[ \Theta\ot(\Pi\op\Sigma)\leq (\Theta\ot\Pi)\op\Sigma \quad{\rm and}\quad (\Theta\op\Pi)\ot\Sigma\leq\Theta\op(\Pi\ot\Sigma)\] since the inequalities hold when they are viewed as linear $(M, N), (N,P)$ and $(P,R)$-monad modules in ${\sf Mon}\MatrQ$. \\

Suppose \QMod\ is a linear quantaloid, then consider the mapping of each object $a\in\cQ$ to the trivial linear \cQ-category $M_a = (1,\rho_a)$ where $1=\{*\}$ is the singleton set, $\rho_a(*)=a$, ${M_a}_\ot = \top_a$ and ${M_a}_\op = \bot_a$, each morphism to $f\colon a\rarr b$ to the trivial linear \cQ-module $\Theta_f\colon M_a\tov M_b$, where $\Theta_f(*,*) = f$. This mapping ensures the linear distributivities hold in \cQ\ as they hold \QMod.
\end{proof}

\begin{rem}
As expected, the constructions are related: let \cQ\ be a linear quantaloid, then $\QMod\cong{\sf Mon}\MatrQ$. Therefore, given any LD-quantale $Q$, we can then consider linear quantaloid ${\sf Mon}(\QRel)$ which is isomorphic to ${\sf Mon}{\sf Matr}\cB(Q)$ and $\cB(Q)$-${\sf Mod}$. 
\end{rem}

\begin{ex} Let $\sREL$ be the Girard quantaloid of sets and relations. The linear monads in \sREL\ are preordered sets $(X,\leq)$ additionally endowed with the relation $\leq^\perp$ defined by 
\[ x\leq^\perp y \iff y\nleq x\]
The linear monad modules are order ideals $R\colon (X,\leq_X)\rarr(Y,\leq_Y)$. In other words, the linear quantaloid of linear monads in \sREL\ is isomorphic to the quantaloid of ordered sets and ideals $\Mon(\sREL)\cong\sORD$. 
\end{ex}

The above example follows from a more general result about Girard quantaloids (and even further cyclic $*$-autonomous categories). Given a linear monad $(a,m_\ot,m_\op)$ in a quantaloid \cQ, the monad maps $m_\ot\colon a\rarr a$ and $m_\op\colon a\rarr a$ are cyclic linear adjoints. Recall that these are unique (up to isomorphism, which is equality in the posetal context) and, as a cyclic $*$-autonomous bicategory, cyclic linear adjoints are canonically given by the $(-)^\perp$. Thus, $m_\op = m_\ot^\perp$. Furthermore, monad modules $f\colon (a,m)\rarr (b,n)$ automatically become linear monad modules $f\colon (a,m,m^\perp)\rarr (b, n_, n^\perp)$.\\

Therefore, the quantaloid of linear monads and linear monad modules of a Girard quantaloid is isomorphic to the quantaloid of monad and monad modules. In other words, considering Girard quantaloids does not give us any truly new quantaloids. 

\begin{ex}
Let \mbox{${\sf P_{max}}$-\sREL} be the linear quantaloid of sets and ``extended'' distance relations. Then consider ${\sf Mon}$(\mbox{${\sf P_{max}}$-\sREL}), a linear quantaloid of sets $X$ endowed with a relation $m_\ot\colon X\times X\rarr [0,\infty]$ satisfying
	\[ m_\ot(x,x)\leq 0 \quad\quad{\rm and}\quad\quad m_\ot(x,x'')\leq {\sf max}(m_\ot(x,x'),m_\ot(x',x'')) \] 
	and relation $m_\op\colon X\times X\rarr [0,\infty]$ satisfying  
	\[ \infty \leq m_\op(x,x) \quad\quad{\rm and}\quad\quad {\sf min}(m_\op(x,x'),m_\op(x',x''))\leq m_\op(x,x'') \] 
	such that 
	\[ {\sf min}(m_\op(x,x'), m_\ot(x',x'')) \leq m_\ot(x,x'')\quad\quad{\rm ,}\quad\quad {\sf min}(m_\ot(x,x'),m_\op(x',x''))\leq m_\ot(x',x'') \] 
	\[ m_\op(x,x'')\leq {\sf max}(m_\ot(x,x'),m_\op(x',x'')) \quad\quad{\rm and}\quad\quad m_\op(x,x'')\leq {\sf max}(m_\op(x,x'),m_\ot(x',x'')) \]
	These are Lawvere ultrametric spaces $(X, m_\ot)$ with an additional distance relation $m_\op$ which interact coherently with $m_\ot$, in particular they are cyclic linear adjoints. The arrows $(X,m)\tov (Y,n)$ are real-valued functions $f\colon X\times Y\rarr[0,\infty]$ such that 
	\[ f(x,y')\leq {\sf max}(f(x,y), n_\ot(y,y') \quad\quad{\rm and}\quad\quad f(x,y)\leq {\sf max}(m_\ot(x,x'), f(x',y))\]
	\[ {\sf min}(f(x,y'), n_\op(y',y))\leq f(x,y) \quad\quad{\rm and}\quad\quad {\sf min}(m_\ot(x,x'), f(x',y))\leq f(x,y) \]
Note that the top two inequalities imply the bottom two and vice versa. We can of course do the same construction with \mbox{${\sf P_{+}}$-\sREL} and \mbox{$[0,1]$-\sREL} and get similar examples of ``linearized'' metric spaces.
\end{ex}

\section{Non-locally posetal examples, $\biloc$, $\biquant$ and $\qtld$}\label{non-posetal}

In this section, we present examples of linear bicategories which are not locally partially ordered. \\

While the bicategory of locales being a linear bicategory will be a consequence of Theorem~\ref{biquant}, we start with case of locales as it is the setting in which these greater results were first investigated. 

\begin{defn}\cite{Joyal_Tierney_1984}\label{Loc}
Let $W, X$ and $Y$ be locales.
\begin{itemize}
	\item An {\em $(X,Y)$-module} $A\colon X\tov Y$ is a sup lattice $A$ which is a left $X$-module and a right $Y$-module satisfying $x(ay)=(xa)y$.
	\item If $B$ is a $(W,Y)$-module, then the sup lattice of right $Y$-module homomorphisms, denoted by $B~\lTo A$, becomes a $(W,X)$-module via $(wf)(a)=wf(a)$ and $(fx)(a)=f(xa)$. Dually, if $C$ is a left $(X,Z)$-module, $A\rTo~C$ is a $(Y,Z)$-module. 
	\item A function $f\colon A\rarr A'$ is a {\em module homomorphism} if it is a left  $X$-module and right $Y$-module homomorphism.
	\item Suppose $A$ is an $(X,Y)$-module. Then the opposite lattice $A^\circ$ becomes a $(Y,X)$-module as follows. Given $x\in X$, the function $x\cdot- \colon A\rarr A$ is sup-preserving map, and hence, has a right adjoint $- /x$. Thus, $A^\circ$ becomes a right $X$-module via $(a,x)\mapsto a/x$ and a left $Y$-module via $(y,a)\mapsto r\backslash a$, where $y\backslash -$ is right adjoint to $-\cdot y$, and $(A^\circ)^\circ\cong A$.
\end{itemize}
\end{defn}

While not explicitly stated, \cite{Joyal_Tierney_1984} develop the required results to say that:

\begin{thm}
There is a bicategory whose objects are locales, 1-cells are modules, and 2-cells are module homomorphisms. Composition of $A\colon X\tov Y$ and $B\colon Y\tov Z$ is given by
\[ A\ot B=A\otimes_Y B\cong (B^\circ \lTo A)^\circ\cong (B\rTo\, A^\circ)^\circ\] 
with identity 1-cells $X\colon X\tov X$. We denote this bicategory by $\biloc$.
\end{thm}

Now, we note that $\biloc$ admits another composition of 1-cells 
\[A\op B=(B^\circ \ot A^\circ)^\circ = (B^\circ\otimes_Y A^\circ)^\circ\cong  A^\circ\rTo\, B\]
with identity 1-cells $X^\circ\colon X\tov X$,  since $X^\circ\op A  \cong A\cong A\op Y^\circ$ and which is associative,
 
\[A\op (B\op C) =  ((C^\circ\ot B^\circ)\otimes  A^\circ)^\circ \cong (C^\circ\ot (B^\circ\ot  A^\circ))^\circ= (A\op B)\op C \]
for all $C\colon Z\tov W$. To see that $\biloc$ is a linear bicategory, we will define the linear distributivity 
\[A\ot (B\op C)\rarr (A\ot B)\op C\]

Since $A\ot (B\op C)\cong A\otimes_Y (B^\circ\rTo\, C)$ and $(A\ot B)\op C \cong (B^\circ\lTo A)\rTo\, C$, it suffices to define a 2-cell $A\ot_Y (B^\circ\rTo\, C) \rarr (B^\circ\lTo A)\rTo\, C$ or equivalently,  
\[ (B^\circ\lTo A) \ot_X (A \ot_Y (B^\circ\rTo\, C)) \rarr C \]
For this, we can use the evaluation maps \[(B^\circ\lTo A) \ot_X (A \ot_Y (B^\circ\rTo\, C)) \cong ((B^\circ\lTo A) \ot_X A) \ot_Y (B^\circ\rTo\, C) \xrightarrow{\epsilon_{A,Y}} B^\circ\ot_Y (B^\circ\rTo\, C) \xrightarrow{\epsilon_{Z,B^\circ}} C\]
Similarly, we get a 2-cell $(A\op B)\ot C\rarr A\op (B\ot C)$, as desired.

\begin{thm}
Under the above operations, $\biloc$ is a linear bicategory.
\end{thm}

\subsection{\cV-{\sf Mat}, \MonB\ and \cV-$\prof$}

Recall $\vmat$, the bicategory of sets, \cV-matrices and \cV-matrix morphisms, which is biclosed when \cV\ is a symmetric monoidal closed category with set-indexed products and coproduct by Lemma \ref{VMat_biclosed}. If \cV\ is further $*$-autonomous, $\vmat$ becomes a linear bicategory, as stated by \cite{Cockett_Koslowski_Seely_2000}:

\begin{prop}\label{VMat_starauto}  
Consider $\vmat$, where \cV\ is a $*$-autonomous category with set-indexed products and coproducts. Given a set $X$, define $\bott_X\colon X\tov X$ by \[\bott_X(x,x')= \begin{cases} \bot & x=x' \\ \top & \text{otherwise}\\ \end{cases}\]
For \cV-matrix $A\colon X\tov Y$, defining $A^\perp\colon Y\tov X$ by $A^\perp(y,x)=A(x,y)^\perp = A(x,y)\rTo\,\bot$, one can show that $(A\rTo\,\bott_Y) \cong A^\perp\cong (\bott_X\,\lTo A)$ and $A\cong (A^\perp)^\perp$,  and so $\vmat$ is a cyclic $*$-autonomous bicategory with cyclic dualizing family $\mathcal{ D}=\{\bott_X\colon X\tov X \}$.
\end{prop}

Now consider \MonB, the bicategory of monads in \cB, modules and module morphisms, which is biclosed if \cB\ is biclosed with local equalizers and coequalizers stable under composition by Lemma \ref{MonB_biclosed}. Then \MonB\ can be a linear bicategory:

\begin{prop}\label{MonB_starauto} 
Suppose \cB\ is a cyclic $*$-autonomous bicategory with local equalizers and coequalizers stable under composition, then \MonB\ is a cyclic $*$-autonomous bicategory.
\end{prop}
\begin{proof} 
Suppose $\cD=\{{\bot_X}\colon X\rarr X \, \vert \, X \in \cB\}$  is a cyclic dualizing family for \cB. Define the family of modules 
\[ \cD^\perp = \{ Q^\perp \colon (X,Q)\tov (X,Q)\,\vert\, (X,Q)\in\MonB \} \]
This family is well-defined since the 1-cell $Q^\perp = Q\rTo\,\bot_X\colon X\rarr X$ is a module $(X,Q)\tov (X,Q)$ with action $\lambda\colon Q\ot Q^\perp\rarr Q^\perp$ defined by the transpose of \[ Q\ot (Q\ot Q^\perp)\xrightarrow{\sim} (Q\ot Q)\ot Q^\perp \xrightarrow{m\ot Q^\perp} Q\ot Q^\perp \xrightarrow{\epsilon_{X,Q}} \bot_X \]
and action $\rho$ defined as $Q^\perp\ot Q\xrightarrow{\sim} (\bot_X\,\lTo Q)\ot Q \xrightarrow{\rho'} \bot_X\,\lTo Q \xrightarrow{\sim} Q^\perp$, where $\rho'$ is the transpose of \[ ((\bot_X\,\lTo Q)\ot Q)\ot Q \xrightarrow{\sim} (\bot_X\,\lTo Q)\ot (Q\ot Q) \xrightarrow{(\bot_X\,\lTo Q) \ot m} (\bot_X\,\lTo Q)\ot Q \xrightarrow{\epsilon_{Q,X}} \bot_X \] 

Given a monad module $A\colon (X,Q)\tov (Y,R)$, applying Proposition~\ref{natural}, we see that \[A\rTo_Q\, Q^\perp = A\rTo_Q (Q\rTo \bot_X) = A\rTo_Q (Q\rTo_{\top_X} \bot_X)\cong A\rTo_{\top_X} \bot_X = A\rTo\,\bot_X = A^\perp \] and \[ R^\perp \lTo_R A \cong (\bot_Y\,\lTo R)\lTo_R A =(\bot_Y\,\lTo_{\top_Y} R)\lTo_R A\cong \bot_Y\,\lTo_{\top_Y} A = \bot_Y\,\lTo A \cong A^\perp \]
Since all these 2-cells are natural in $A$, $\cD^\perp$ is a cyclic dualizing family for \MonB.
\end{proof}

Finally, recalling that \cV-$\prof$ can be constructed as ${\sf Mon}(\vmat)$ by Proposition \ref{BMod_as_MonMatrB}, and by the above two results:
\begin{prop}\label{VProf_starauto}
If \cV\ is a complete and cocomplete $*$-autonomous category, then \cV-$\prof$ is a cyclic $*$-autonomous bicategory.
\end{prop}

\subsection{Linear bicategories $\biquant$ and $\qtld$}

\cite{Joyal_Tierney_1984} observed that for a fixed commutative unital quantale $Q$, the category of right $Q$-modules and module homomorphisms is a $*$-autonomous category. \cite{Rosenthal_1994} expanded this result to show that the category of $(Q,Q)$-modules is equally $*$-autonomous and that, given a small quantaloid \cQ, the category of $(\cQ,\cQ)$-modules is $*$-autonomous.\\

Now that we have access to the notion of linear bicategories, $Q$ and \cQ\ no longer need to be fixed and we can claim the following examples of cyclic $*$-autonomous bicategories.

\begin{defn}
\begin{itemize}
	\item Given unital quantales $Q$ and $R$, a $(Q,R)$-module $Q\tov R$ is a suplattice $A$ which is a left $Q$-module and a right $R$-module, i.e., there are suplattice homomorphisms $\star\colon Q\times A\rarr A$ and $\cdot\colon A\times R\rarr A$ such that 
	\[ (q \ot q')\star a = q\star(q'\star a) \quad, \quad a\cdot(r\ot r') = (a\cdot r)\cdot r'  \]
	 \[ \top \star a = a\quad,\quad a\cdot \top = a \quad{\rm and}\quad (q \star a)\cdot r = q\star(a\cdot r) \]

	\item Given $(Q,R)$-modules $A$ and $B$, a module homomorphism is a suplattice homomorphism $f\colon A\rarr B$ such that
	\[f(q\star a) = q\star f(a) \quad\quad{\rm and}\quad\quad f(a\cdot r) = f(a)\cdot r \]
	
	\item Let $\biquant$ denote the bicategory of unital quantales, modules and module morphisms. Then, $\biquant$ is ${\sf Mon}\mathcal{ B}({\rm Sup})$, the bicategory of monoids, monoid modules and module morphisms in ${\rm Sup}$. 
\end{itemize}

\end{defn}

If \cV\ is $*$-autonomous category with equalizers and coequalizers, its suspension $\cB(V)$ is a cyclic $*$-autonomous bicategory with local equalizers and co-equalizers stable under composition. Then, by Proposition \ref{MonB_starauto}, the bicategory ${\sf Mon}(\cB(V))$ is a cyclic $*$-autonomous bicategory. Taking $\cV={\rm Sup}$, we get:

\begin{thm}\label{biquant}
The bicategory $\biquant$ is a cyclic $*$-autonomous bicategory.
\end{thm}

\begin{defn}
\begin{itemize}
	\item Given small quantaloids $\cQ$ and $\mathcal{ R}$, a $(\cQ,\mathcal{ R})$-module $A\colon \cQ\tov\mathcal{ R}$ consists of, for each $q, q'\in{\rm ob}\cQ, r, r'\in{\rm ob}\mathcal{ R}$:
	\begin{itemize}
		\item a suplattice $A(q,r)$, 
		\item a left action suplattice homomorphism $\star\colon \cQ(q,q')\times A(q',r)\rarr A(q,r)$ such that given $a\in A(q,r)$
		 \[ \top_q \star a = a \quad {\rm for}\quad q=q' \quad\quad (f\ot g)\star a = f\star (g\star a) \quad {\rm for}\quad f\colon q\rarr q', g\colon q'\rarr q'' \]
		\item a right action suplattice homomorphism $\cdot\colon A(q,r)\times \mathcal{ R}(r,r')\rarr A(q,r')$ such that given $a\in A(q,r)$
		\[  a\cdot\top_r = a \quad {\rm for}\quad r=r' \quad\quad a\cdot(h\ot k) = (a\cdot h)\cdot k\quad {\rm for}\quad h\colon r\rarr r', k\colon r'\rarr r'' \]
	\end{itemize}
	satisfying $\forall a\in A(q,r), f\colon q\rarr q' \in \cQ, h\colon r\rarr r' \in \mathcal{ R}$, \[ (f\star a)\cdot h = f\star (a\cdot h) \]
	
	\item Given $(\cQ,\mathcal{ R})$-modules $A$ and $B$, a module homomorphism $f\colon A\rarr B$ is a family of suplattice homomorphisms $f_{q,r}\colon A(q,r)\rarr B(q,r)$ satisfying, for $f\colon q\rarr q' \in\cQ, a'\in A(q',r), a\in A(q,r), h\colon r\rarr r' \in \mathcal{ R}$,
	\[f_{q,r}(f\star a') = q\star f_{q',r}(a') \quad\quad{\rm and}\quad\quad f_{q,r'}(a\cdot h) = f_{q,r}(a)\cdot h \]
	
	\item Let $\qtld$ denote the bicategory of small quantaloids, modules and module homomorphisms. Then, $\qtld$ is ${\rm Sup}$-$\prof$, the bicategory of ${\rm Sup}$-categories, ${\rm Sup}$-profunctors and ${\rm Sup}$-transformations.
\end{itemize}
\end{defn}

By Proposition \ref{VProf_starauto} and taking $\cV={\rm Sup}$, we get:
\begin{thm}
The bicategory $\qtld$ is a cyclic $*$-autonomous bicategory.
\end{thm}

\textbf{Acknowledgments.} The authors would like to thank an attentive anonymous referee for making many suggestions which improved the article, in particular bringing our attention to the notion of bimonoids in the algebraic logic literature, and to the article \cite{Galatos_Prenosil_2023}, and for calling attention to bi-Heyting algebras and the opposite infinitary law when considering locales. The first and second authors also acknowledge the support of the Natural Sciences and Engineering Research Council of Canada (NSERC), under the grant awarded to Richard Blute.

\bibliographystyle{plain} 
\bibliography{bibliography}

\end{document}